%% file: main-arxiv.tex
\documentclass[11pt]{article}
\usepackage[letterpaper,margin=1in]{geometry}

\input{macros-arxiv.tex}

\title{On Connectivity in Random Graph Models with Limited Dependencies}

\let\anonymous\undefined %

\ifdefined\anonymous
    \author{Anonymized}
    \affil{Affiliation\\ \texttt{email@adre.ss}}
\else
\author[1]{Johannes~Lengler}
\author[1]{Anders~Martinsson}
\author[1]{Kalina~Petrova}
\author[1]{Patrick~Schnider}
\author[1]{Raphael~Steiner}
\author[1]{Simon~Weber}
\author[1]{Emo~Welzl}
\affil[1]{Department of Computer Science, ETH Z\"urich}
\fi

\date{}

\begin{document}

\maketitle

\begin{abstract}
    For any positive edge density $p$, a random graph in the \erdos-\renyi{} $G_{n,p}$ model is connected with non-zero probability, since all edges are mutually independent. We consider random graph models in which edges that do not share endpoints are independent while incident edges may be dependent and ask: what is the minimum probability $\rho(n)$, such that for any distribution $\mathcal{G}$ (in this model) on graphs with $n$ vertices in which each potential edge has a marginal probability of being present at least $\rho(n)$, a graph drawn from $\mathcal{G}$ is connected with non-zero probability?
    
     As it turns out, the condition ``edges that do not share endpoints are independent'' needs to be clarified and the answer to the question above is sensitive to the specification. In fact, we formalize this intuitive description into a strict hierarchy of five independence conditions, which we show to have at least three different behaviors for the threshold $\rho(n)$. For each condition, we provide upper and lower bounds for $\rho(n)$. In the strongest condition, the \emph{coloring model} (which includes, e.g., random geometric graphs), we show that $\rho(n)\rightarrow 2-\phi\approx 0.38$ for $n\rightarrow\infty$, proving a conjecture by Badakhshian, Falgas-Ravry, and Sharifzadeh. This separates the coloring models from the weaker independence conditions we consider, as there we prove that $\rho(n)>0.5-o(n)$.
     In stark contrast to the coloring model, for our weakest independence condition --- pairwise independence of non-adjacent edges --- we show that $\rho(n)$ lies within $O(1/n^2)$ of the threshold $1-2/n$ for completely arbitrary distributions.
\end{abstract}

\vfill
{
\small
\paragraph{Acknowledgments.}
\ifdefined\anonymous

\else
This research started at the joint workshop of the \emph{Combinatorial Structures and Algorithms} and \emph{Theory of Combinatorial Algorithms} groups of ETH Z\"urich held in Stels, Switzerland, January 2023.
We thank the organizers for providing a very pleasant and inspiring working atmosphere. K.P. was supported by grant no. CRSII5 173721 of the Swiss National Science Foundation. R.S. was supported by an ETH Z\"urich Postdoctoral Fellowship. S.W. was supported by the Swiss National Science Foundation under project no. 204320. Last but not least, we would like to thank Victor Falgas-Ravry for pointing out an abundance of related work, thus helping tremendously to put our results in perspective.
\fi
}

\newpage
\section{Introduction}

The probabilistic method is an important tool in theoretical computer science, graph theory and combinatorics~\cite{alon2016probabilistic}. With this method, one proves that a random construction has some desirable property with positive probability, and concludes that some objects with this property must exist. Early examples of the probabilistic method include simple constructions of expander graphs~\cite{BOLLOBAS1988241}, or graphs with high girth and large chromatic number~\cite{erdos_1959}. Often, it is possible to express the property as the intersection of very simple events $A_1,\ldots,A_k$. For example, we can express the property that a set $S$ is a clique as the intersection of the events ``$u$ is adjacent to $v$'' for all pairs $u,v\in S$. In many cases it is clear that each of the $A_i$ has a large probability to occur, but this is generally not enough to conclude that the intersection of all $A_i$ has positive probability to occur. Of course, if the probability space were a \emph{product space} in which all components are \emph{independent} of each other, then this would be trivial.

However, for many applications it is too limiting to restrict oneself to product spaces. The probabilistic method was significantly extended by the realization that it could still be applied to settings without perfect independence, provided there is some bound on the amount of dependence in the system. The seminal result was the Lov\'asz Local Lemma (LLL), which we give here in the slightly stronger version due to Shearer~\cite{shearer1985problem}: for events $A_1,\ldots,A_k$ that all occur with probability at least $p$, if each of the $A_i$ depends on at most $d$ other $A_j$, where $p \ge 1-1/(ed)$, then there is a positive probability that all of the $A_i$ occur simultaneously.\footnote{\label{footnote:Shearer}The precise condition is $p > 1-\frac{(d-1)^{(d-1)}}{d^d}$ for $d>1$ and $p> 1/2$ for $d=1$. The first expression can be simplified by the estimate $\frac{(d-1)^{(d-1)}}{d^{d-1}} \ge e^{-1}$, which becomes tight in the limit $d \to \infty$.} In this context, the dependencies are captured by a \emph{dependency graph}, a graph with vertex set $\{A_1,\ldots,A_k\}$ such that each vertex is independent from all but its neighbors. There are many situations in which the dependency graph can be restricted, and we give some examples in Section~\ref{sec:bounded-dependency-graphs} below.

The LLL allows to re-introduce a product space through the backdoor, because the LLL condition $p \ge 1-1/(ed)$ allows to couple the process to a product space. This was already implicit in the inductive proof of the LLL~\cite{spencer1977asymptotic}, and was made explicit by Liggett, Schonmann and Spacey~\cite{liggett1997domination}. They also generalized this coupling to a countably infinite number of variables and showed tightness of Shearer's condition (of the precise version in footnote~\ref{footnote:Shearer}). For a finite number of variables, coupling the probability space to a product space implies trivially that the all-one event has positive probability.

Unfortunately, the LLL scales badly in $d$, i.e., $p$ needs to be very close to one if $d$ is large. However, in some cases the degree $d$ may grow. In particular, in this paper we will study the situation that the variables are associated with the edges of a complete graph on $n$ vertices, and dependencies only run between adjacent edges (edges which share a common endpoint). This is a common situation, and examples are given in Section~\ref{sec:bounded-dependency-graphs}. In this case, the number of dependencies per edge is $2(n-2)$, and thus grows with $n$. Thus, if every edge is present with constant probability $p<1$, then it is not possible to couple the probability space with a product space for large $n$, and it is not true that the complete graph appears with positive probability~\cite{liggett1997domination}. However, in this paper we will show that some weaker global properties \emph{can} be guaranteed. Specifically, we will focus on connectivity because this is arguably the most fundamental global graph property. We will study the question:\medskip

\parbox{0.92\textwidth}{
Consider a random graph in which every potential edge is inserted with probability at least $p$. Assume that non-adjacent edges are independent. For which values of $p$ can we guarantee that the graph is connected with positive probability? 
}\medskip

\noindent It turns out that the answer depends heavily on what exactly we mean by independence. Surprisingly, there are at least five different ways to interpret the innocent-looking condition that non-adjacent edges are independent, which we define next in order of increasing strength. For an edge $e=\{u,v\}$, let $X_e$ be the event that $e$ is present in the random graph.

\begin{definition}[Pairwise independence]\label{def:pairwise}
For every pair of non-adjacent edges $e,f$, the random variables $X_e$ and $X_f$ are independent.
\end{definition}

\begin{definition}[Matching independence]\label{def:matching}
    For any set $M$ of pairwise non-adjacent edges (also called a \emph{matching}), $\{X_e\mid e\in M\}$ are mutually independent.
\end{definition}

\begin{definition}[Edge-subgraph independence]\label{def:edge-subgraph} For any edge $e=\{u,v\}$, any vertex set $W$ with $u,v \notin W$, and any choice of graph $H_W$ on $W$, we have for the random graph $G$ sampled from the distribution that \[Pr[G[W] = H_W \textrm{ and } X_e] = Pr[G[W]=H_W]\cdot Pr[X_e].\]
\end{definition}

\begin{definition}[Subgraph independence]\label{def:subgraph}
    For any two disjoint subsets $V,W$ of vertices, and any choice of graphs $H_V$ and $H_W$ on $V$ and $W$, respectively, we have for the random graph $G$ sampled from the dsitribution that
    \[Pr[G[V]=H_V\textrm{ and } G[W]=H_W] = Pr[G[V]=H_V]\cdot Pr[G[W]=H_W].\]
\end{definition}

\begin{definition}[Coloring model]\label{def:coloring}
The graph distribution is given by a set of probability spaces $(\Omega_v,Pr_v)$, one for each vertex $v\in V$, and a set of deterministic functions $f_{u,v}:\Omega_u\times \Omega_v\rightarrow \{0,1\}$ computing $X_{\{u,v\}}$, one for each edge $\{u,v\}\in \binom{V}{2}$. If every probability space $(\Omega_v,Pr_v)$ is finite, we call $\max_v|\Omega_v|$ the \emph{number of colors} of the coloring model.
\end{definition}

We remark that the standard proof of LLL requires edge-subgraph independence. In the paper~\cite{liggett1997domination} that makes the coupling to product spaces explicit, the authors describe subgraph independence as the required property, but inspecting the proof shows that they only use the weaker condition of edge-subgraph independence.

\begin{table}[bp]
    \centering
    \begin{tabular}{l|c|c}
        Independence condition & Lower bound & Upper bound \\
        \hline 
        Coloring model (2 colors) & $1/4$ & $1/4$ \\
        Coloring model (general) & $0.381966\ldots$ & $0.381966\ldots$ \\
        Subgraph independence & $1/2$~\cite[Thm. 16]{day2020oneindependent} & $1/2$~\cite[Thm. 16]{day2020oneindependent} \\
        Edge-subgraph independence & $1/2$ & $3/4$ \\
        Matching independence & $1/2$ & $1$ \\
        Pairwise independence & $1$ & $1$
    \end{tabular}
    \caption{Lower and upper bounds for the threshold $\rho$ s.t. every constant edge probability $>\rho$ guarantees connectivity with positive probability for all sufficiently large $n$, while a constant edge probability $<\rho$ does not guarantee connectivity with positive probability for infinitely many $n$.}
    \label{tab:bounds-simple}
\end{table}

Note that some of our models have been previously studied under different names. We summarize the bounds on the connectivity thresholds obtained before or simultaneously to this paper in~\Cref{fig:rho_intervals_before}, and discuss this related work in more detail in \Cref{sec:bounded-dependency-graphs}. Our main results for large~$n$ are summarized in Table~\ref{tab:bounds-simple} and \Cref{fig:rho_intervals}. We give a refined exposition of our results for \emph{all} $n$ in \Cref{thm:main}. We do not have matching upper and lower bounds in all cases, but as can be seen in \Cref{fig:rho_intervals}, our bounds imply that there are at least three different thresholds among the five independence conditions discussed above, and four different thresholds if we also include the coloring models with only $2$ colors. Notably, our proofs give different intervals for the thresholds of all six independence conditions, although this does not imply that the thresholds are all different. 

\begin{figure}[ht]
    \centering
    \includegraphics{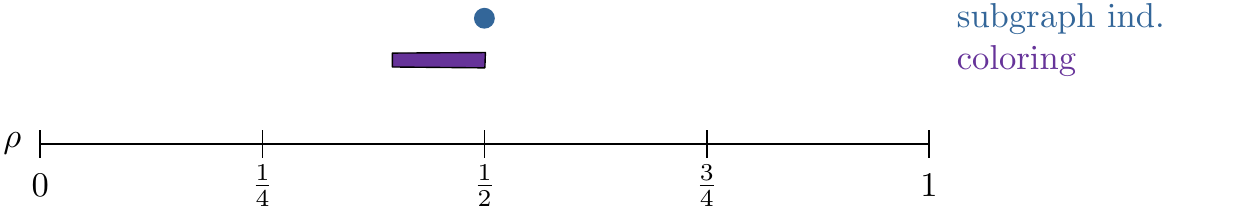}
    \caption{An illustration of the upper and lower bounds on the thresholds known from related work. The bounds for subgraph independence can be found in \cite[Thm. 16]{day2020oneindependent}. The bounds on the coloring model have been obtained simultaneously and independently to ours~\cite[Thm. 1.7]{badakhshian2023density}.}
    \label{fig:rho_intervals_before}
\end{figure}

\begin{figure}[ht]
    \centering
    \includegraphics{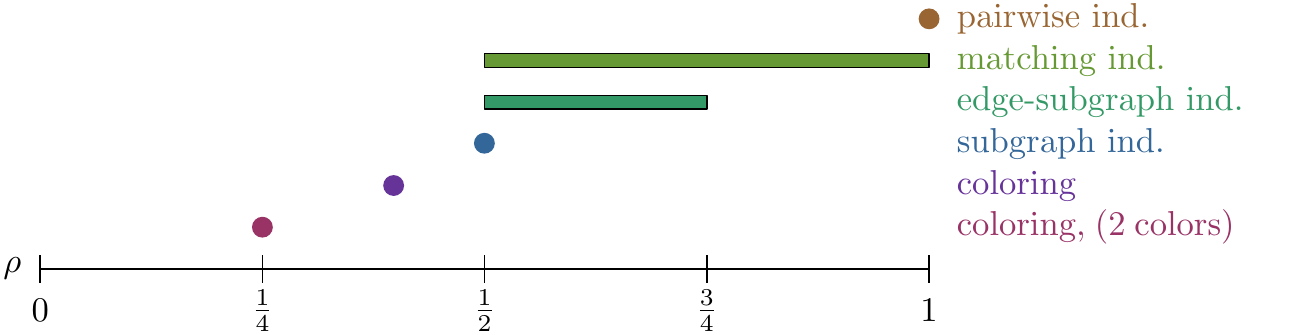}
    \caption{An illustration of the upper and lower bounds mentioned in \Cref{tab:bounds-simple} and \Cref{thm:main}. As can be seen, there must be at least four thresholds $\rho$ among the six independence conditions.}
    \label{fig:rho_intervals}
\end{figure}

Moreover, let us write $\calG_{\text{pw}}(n,p)$ for the class of all random graph distributions on $n$ vertices, with marginal edge probabilities at least $p$ and with the property ``pairwise independence''. We write $\calG_{\text{pw}} := \bigcup_{n\in \N} \calG_{\text{pw}}(n,0)$ for the same class without restrictions on the number of vertices and the marginal probabilities. Likewise, we write $\calG_{\text{mat}}$, $\calG_{\text{esub}}$, $\calG_{\text{sub}}$ and $\calG_{\text{col}}$ for the random graph distributions satisfying matching independence, edge-subgraph independence, subgraph independence and the coloring models respectively. We show that those five models form a strict hierarchy.

\begin{restatable}{theorem}{hierarchy}\label{thm:hierarchy}
$\calG_{\text{col}} \subsetneq \calG_{\text{sub}} \subsetneq \calG_{\text{esub}} \subsetneq \calG_{\text{mat}} \subsetneq \calG_{\text{pw}}$.
\end{restatable}

\subsection{Related Work and Examples}\label{sec:bounded-dependency-graphs}
There are numerous applications in graph theory, theoretical computer science, and combinatorics in which dependency graphs are bounded or otherwise restricted. For example, if we want to find a vertex coloring in a hypergraph without monochromatic hyperedges, and we color the vertices independently, then any two disjoint hyperedges are independent. Therefore, often times LLL type arguments apply to coloring (hyper)graphs of bounded maximum degree. For instance, the seminal results by Johansson~\cite{johansson} and Molloy~\cite{molloy} stating that every graph of maxmimum degree $\Delta$ has (list)
chromatic number $O(\frac{\Delta}{\log \Delta})$ were proved using this method. Similarly, if we want to find a satisfying assignment to a SAT-formula in conjunctive normal form, and assign the variables independently, then any two clauses are independently satisfied if they do not share a common literal, see for instance~\cite{Gebauer2009} for more detail and background. These are classical applications of LLL and fall into the class of coloring models.

Another important source of applications comes from percolation theory, which was also the main motivation for~\cite{liggett1997domination}.
There, an important technique to study locally generated geometric graphs is \emph{rescaling}: the geometric space is covered by boxes that may partially overlap, and boxes are called \emph{good} under some conditions that depend only on the subgraph induced by the box. It is shown recursively that the probability for a box to be good increases with the box size. The goodness of two boxes is independent if they don't overlap, and typically every box only overlaps with constantly many other boxes. 

So far, we gave general examples of bounded dependence. We now turn to examples where the random variables are naturally tied to the edge set of a complete graph. An important class of examples come from random graph models that are more complex than Erd\H os-R\'enyi graphs, in particular models that capture graph properties of real-world networks like clustering, communities, or heavy-tailed degree distributions. Many such models are generated by drawing some information for each vertex, and connecting two vertices based on this information. There is a large number of such random graph models: in Chung-Lu or Norros-Reittu random graphs, the vertices draw weights which determine their expected degrees~\cite{chung2002connected,reittu2012random}; in Random Geometric Graphs (RGG) or Hyperbolic Random Graphs (HRG) the vertices draw some location in a geometric space~\cite{krioukov2010hyperbolic,penrose2003random}; in scale-free percolation (SFP) the vertices lie on a grid and also draw random weights~\cite{deijfen2013scale}; in Geometric Inhomogeneous Random Graphs (GIRG), they draw both a position and a weight~\cite{bringmann2019geometric}. In the Stochastic Block Model (SBM) they draw the community to which they should belong~\cite{karrer2011stochastic}. Many more applications of similar flavour that cannot all be listed here can be found in the literature. All these models fall in the class of coloring models.\footnote{The formulation of many models involves a coin flip for each edge, where the probability of inclusion is a function of the information of the two endpoints. In this formulation, the event $X_e$ is not a \emph{deterministic} function of the information. But one can simply define the coin flip as part of the random experiments of one of the endpoints.} 

The class of coloring models is the most important one.\footnote{But note that it is not the model that arises naturally for the LLL. This raises the question whether stronger LLL-type results can be obtained for the coloring model.} Models that do not fall into this class can arise when the graphs drawn from the distribution must fulfill some global property. For example, consider the following distribution for an \emph{odd} number of vertices $n\geq 5$. Every vertex chooses a color from red and blue independently and uniformly at random. The resulting graph consists only of a clique on the color chosen an even number of times. We show later (in the proof of \Cref{lem:esub_mat_inclusion}) that this distribution is matching independent, but not edge-subgraph independent (and thus also not a coloring model). As a second example, consider the \erdos-\renyi model $G_{n,1/2}$ with the additional side constraint that the total number of edges $|E|$ is divisible by three.\footnote{More formally, we consider the conditional probability distribution obtained by conditioning the $G_{n,1/2}$ distribution on the event that $|E|\equiv_3 0$.} One can show quite easily that this distribution is not even pairwise independent. On the other hand, if we instead use the side constraint that $|E|$ is divisible by \emph{two}, the distribution is actually a coloring model, indicating that independence conditions are surprisingly fickle. A proof of both of these facts can be found in \Cref{app:GnpSideConstraints}.

Two of our independence conditions have been previously considered under different names. Firstly, subgraph independent distributions have been studied under the name of \emph{$1$-independent} random graphs. More generally, a $k$-independent graph distribution is a distribution where any two sets of edges $E,F$ are independent if the minimum distance between any vertex incident to an edge in $E$ and any vertex incident to an edge in $F$ is at least~$k$. These $k$-independent graph distributions have been studied extensively in the context of percolation theory, for example on the infinite integer grid~\cite{balisterBollobas}. In~\cite{day2020oneindependent}, Day, Falgas-Ravry, and Hancock consider $1$-independent random \emph{finite} graphs. Among other results, they pin down the precise value of the connectivity threshold for $1$-independent (i.e., subgraph independent) random graphs, as mentioned above in \Cref{tab:bounds-simple} and \Cref{fig:rho_intervals_before}.

Coloring models have been studied under two different names. First off, they have been studied as \emph{vertex-based measures} in the context of percolation in \cite{day2020oneindependent}. Second off, a deterministic perspective on coloring models lies at the core of the well-studied \emph{density \turan problem for multipartite graphs}~\cite{balister2022improved,bondy2006density,csikvari,falgas2021,nagy2011multipartite}: instead of a random graph distribution, the density \turan problem for multipartite graphs considers $|V(G)|$-partite blow-ups $G^*$ of a graph $G$, and the minimum density needed between parts of $G^*$ to guarantee the existence of a graph $H$ as a \emph{transversal}, i.e., as a subgraph of~$G^*$ that picks precisely one vertex per part of $G^*$. This problem has also been considered for the more general case where we are not considering the occurrence of a single graph $H$ as a transversal, but of any graph from some collection $\mathcal{H}$. Furthermore, one can consider the \emph{weighted} case, where the vertices in each part of $G^*$ each get a positive weight such that all the weights per part add up to $1$. The density between two parts is then also computed in a weighted fashion. Now, considering $G$ as the complete graph $K_n$, and $\mathcal{H}$ as the family of spanning trees on $n$ vertices, one can see that the weighted density \turan problem for multipartite graphs is equivalent to the problem of determining the minimum edge probability needed to guarantee connectivity with non-zero probability in coloring models (except that the weighted density \turan setting only allows for modeling finitely many colors, which turns out to be without loss of generality, as we discuss in the following subsection). In~\cite{badakhshian2023density}, which appeared simultaneously and independently from this paper, Badakhshian, Falgas-Ravry, and Sharifzadeh consider exactly this question; they prove the same lower bound as us~\cite[Thm.~1.7]{badakhshian2023density}, and propose our upper bound as a conjecture~\cite[Conj.~1.9]{badakhshian2023density}. 

Last but not least, let us mention that in a conceptually similar (but concretely rather different) direction, Alon and Nussboim~\cite{alonKwise} studied thresholds for the connectivity of random graph models in which only edge-sets of size at most $k$ for a fixed parameter $k$ are required to be independent, but dependencies between edge-sets may occur from size $k+1$ and up. 

\subsection{Detailed Results}\label{sec:detailed-results}
In this section we give the main theorem with more detailed results.
We denote by
\[\rho_{\text{pw}}(n) := \inf\{p\in[0,1] \mid  \forall \calD \in \calG_{\text{pw}}(n,p), \Pr[G_n \sim \calD \text{ is connected}] > 0\}
\] the smallest (infimum) marginal edge probability that guarantees a positive probability for $G_n$ being connected. The quantities $\rho_{\text{mat}}(n)$, $\rho_{\text{esub}}(n)$, $\rho_{\text{sub}}(n)$, and $\rho_{\text{col}}(n)$ are defined similarly. Furthermore, let $\calG_{\text{col},k}$ denote the coloring models with at most $k$ colors, and let $\rho_{\text{col},k}(n)$ be the corresponding threshold. %

\begin{theorem}\label{thm:main}
We have the following bounds on $\rho$:
\begin{align*}
    \intertext{\textbf{Pairwise independence:} For all $n\in \N$,}
    1-2/n -\Theta(1/n^2)\le\; &\rho_{\text{pw}}(n) \le 1-2/n. \\
    \intertext{\textbf{Matching independence:} For all $n\in \N$, }
    \frac{1}{2}(1-\tan^2 \frac{\pi}{2n})\le\; &\rho_{\text{mat}}(n) \le 1-2/n. \\
    \intertext{
    Moreover, for any $k \in \mathbb{N}$,
    }
    \frac{1}{2}\le\; &\rho_{mat}(8k).\\
    \intertext{\textbf{Edge-subgraph independence:} For all $n\in \N$, }
    \frac{1}{2}(1-\tan^2 \frac{\pi}{2n})\le\; &\rho_{\text{esub}}(n) \le 3/4. \\
    \intertext{\textbf{Subgraph independence:} \cite[Thm 16]{day2020oneindependent} For all $n \ge 2$, }
    &\rho_{\text{sub}}(n)=\frac{1}{2}(1-\tan^2 \frac{\pi}{2n}).
    \intertext{\textbf{Coloring model:} For all $n\geq 3$, we have $1/4\le \rho_{\text{col}}(n) \le 1/2$. Moreover, let $\phi = \tfrac{1}{2}(1+\sqrt{5})$ be the golden ratio. Then }
    \lim_{n\to\infty}&\rho_{\text{col}}(n) = 2-\phi \approx 0.381966. \\
    \intertext{\textbf{Coloring model, two colors:} For all $n\geq 3$, $\rho_{\text{col},2}(n) = 1/4$.}
\end{align*}
\end{theorem}
To the best of our knowledge, this question had not been considered so far for the pairwise independence, matching independence, or edge-subgraph independence models. For the subgraph independence model, the threshold was known precisely~\cite{day2020oneindependent}, but we include it in our theorem for completeness of the hierarchy of our models. Our results for the coloring model confirm a conjecture of Badakhshian, Falgas-Ravry, and Sharifzadeh~\cite{badakhshian2023density}.

We remark that, while $\calG_{\text{col}}(n,p)$ includes models with infinitely many outcomes of the random experiment at each vertex, it has been shown\footnote{In fact, these results are stated only for finitely many colors as they are in the density \turan setting, but the proofs go through in our more general set-up as well.} (in~\cite{bondy2006density} for $n=3$ and in~\cite[Lemma 2.1]{nagy2011multipartite} for all $n$) that $\rho_{\text{col},k}(n) \geq \rho_{\text{col}}(n)$ for $k=n-1$, which combined with the trivial $\rho_{\text{col},k}(n) \leq \rho_{\text{col}}(n)$ for any $k$ gives $\rho_{\text{col},n-1}(n) = \rho_{\text{col}}(n)$. The proofs of $\rho_{\text{col},n-1}(n) \geq \rho_{\text{col}}(n)$ take a model which has probability $0$ of being connected and modify it to use at most $n-1$ many colors while preserving the fact that it is never connected and increasing or keeping the same its marginal edge probabilities. However, the probability distribution over all graphs on~$n$ vertices that the modified model gives is usually different from the original one. Here, we prove the even stronger statement that $\calG_{\text{col}}(n,p) = \calG_{\text{col},k}(n,p)$ for some $k = k(n)$, showing that finitely many colors suffice to model any probability distribution over all graphs that can be represented as a coloring model. We defer the proof of \Cref{lem:finitely_many_colors} to~\Cref{app:finitely_many_colors}.
\begin{lemma}
\label{lem:finitely_many_colors}
Consider a coloring model on $n$ vertices along with its corresponding probability distribution $\calD$ over the graphs on $n$ vertices. Then there is a coloring model on $n$ vertices with at most $2^{\binom{n}{2}} + 1$ colors per vertex which results in the same distribution $\calD$.
\end{lemma}

\subsection{Open Questions}

There are many interesting questions that remain open. Most obviously, for large $n$ there are gaps between the upper and lower bounds for matching independence and edge-subgraph independence. For the coloring model we do have matching upper and lower bounds for an unbounded number of colors and for $k=2$, but not for other constant values of $k$. Moreover, for fixed $n$ it is unclear how much richer the coloring model gets by adding more colors. As mentioned above, we know that $\calG_{\text{col}}(n,p) = \calG_{\text{col},k}(n,p)$ with $k=2^{\binom{n}{2}}+1$ and  $\rho_{\text{col}}(n) = \rho_{\text{col},{n-1}}(n)$. But what is still an open question is the behavior of the functions $k^{\calG}_{min}(n) = \min\{k\;|\;\calG_{\text{col}}(n,p) = \calG_{\text{col},k}(n,p)\}$ and $k^{\rho}_{min}(n) = \min\{k\;|\;\rho_{\text{col}}(n,p) = \rho_{\text{col},k}(n,p)\}$, that is, how many colors are enough to be able to express all coloring models on $n$ vertices, respectively to capture the behaviour of the connectivity threshold of the coloring models on~$n$ vertices?

Finally, in this paper we focus on connectivity because that is arguably the most fundamental global property of a graph. We only see this as a starting point and we would find it interesting to explore analogous questions for other global properties. A natural extension might be to study the size of the largest connected component that can be achieved with non-zero probability, but the same questions arise for any other global graph properties such as Hamiltonicity, the chromatic number, and many more.

\subsection{Proof Techniques}
We prove the strict hierarchical structure of our considered independence conditions (\Cref{thm:hierarchy}) by giving concrete examples of distributions that fulfill the weaker independence condition, but not the stronger one. All of these examples are simple to describe and quite illustrative.

The lower bounds on the thresholds $\rho$ are shown by concrete series of graph distributions on disconnected graphs. Here, the lower bound on $\rho_{\text{pw}}$ (pairwise ind.) uses the same distribution as the distribution used to show $\calG_{\text{pw}}\neq \calG_{\text{mat}}$. The lower bounds on $\rho_{\text{mat}},\rho_{\text{esub}}$, and $\rho_{\text{sub}}$ all use the same construction, since we were not able to make use of the additional freedom available in the edge-subgraph or even the matching independence condition (except for the case when $n$ is divisible by $8$ in the matching independence model). This lower bound on $\rho_{sub}$ has been previously proven in~\cite[Thm. 16]{day2020oneindependent}. We provide the construction for completeness. It uses a distribution that can be seen conceptually as a version of a coloring model with \emph{complex probabilities}. The proof that the distribution fulfills subgraph independence follows from the fact that the coloring models are subgraph independent, and from the fundamental theorem of algebra. The lower bounds on $\rho_{\text{col}}$ and $\rho_{\text{col},2}$ once again are simple constructions.

For the upper bounds on the thresholds $\rho$ we use very different proof techniques depending on the independence condition. The upper bound for pairwise and matching independence does not make use of these independence conditions, and just combines a linearity of expectation argument with the maximum number of edges in a disconnected graph. As mentioned above, edge-subgraph independence is the first independence condition which allows us to apply \lovasz Local Lemma to achieve a constant bound on $\rho$.
For the coloring models, we give concrete strategies on how to pick a color for each vertex such that the graph is connected. For the two color case, this strategy is rather simple, while the strategy for the general case is based on adjusting a random coloring, which fulfills useful properties with high probability for large $n$. The proof of this bound (\Cref{thm:coloring_upper_n_infinity}) is by far the most technically involved proof in this paper.

\subsection{Paper Overview}
In \Cref{sec:preliminaries}, we introduce some classical tools from probability theory used in our proofs. Then, in \Cref{sec:hierarchy}, we prove \Cref{thm:hierarchy}, i.e., that our five models of independence form a strict hierarchy. \Cref{sec:bounds} contains the proofs for all bounds summarized in \Cref{thm:main}.

\section{Preliminaries}\label{sec:preliminaries}
We make use of the following version of the \lovasz Local Lemma.
\begin{lemma}[\cite{shearer1985problem}]\label{lem:LLL}
    Let $A_1,\ldots A_k$ be a sequence of events, such that each event occurs with probability at least $p$ and is independent of all the other events, except at most $d$ of them. Then, if
    \[ p > \begin{cases}
    1-\frac{(d-1)^{d-1}}{d^d} & \text{for $d\geq2$,}\\
    $1/2$ & \text{for $d=1$,}
    \end{cases}\]
    there is a non-zero probability that all of the events occur.
\end{lemma}

The following lemma is a direct consequence of Markov's inequality.

\begin{lemma}\label{lem:markovderivate}
    Let $X$ be a random variable such that $Pr[X\leq u]=1$. Then for $\ell < u$, $$Pr[X\geq \ell]\geq \frac{\E[X]-\ell}{u-\ell}.$$
\end{lemma}
\begin{proof}
    Note that the random variable $u-X$ is non-negative and has expectation $u-\E[X]$. We can thus apply Markov's inequality.
    \[Pr[X\geq\ell] = Pr[u-X\leq u-\ell] = 1-Pr[u-X>u-\ell]\geq 1-\frac{u-\E[X]}{u-\ell}=\frac{\E[X]-\ell}{u-\ell}.\qedhere\]
\end{proof}

We also use Azuma's inequality, as for example stated in \cite{janson2011azuma} and originally proven by Azuma~\cite{azuma1967ineq}:
\begin{theorem}\label{thm:azuma}
    Let $(\Omega,Pr)$ be a probability space given as the product of $N$ discrete probability spaces $\{(\Omega_i,Pr_i)\}_{1\leq i\leq N}$, and let $X:\Omega\rightarrow \mathbb{R}$ be a random variable such that the effect of the $i$-th coordinate is at most $c_i$ (i.e., $|X(\omega_1,\ldots,\omega_i,\ldots \omega_N)-X(\omega_1,\ldots,\omega'_i,\ldots,\omega_N)|\leq c_i$ for all $\omega_j\in\Omega_j$ for each $j\neq i$ and all $\omega_i,\omega'_i\in\Omega_i$). Then, for all $t\geq 0$, we have
    \[Pr[X\geq \E[X]+t]\leq e^{-\frac{t^2}{2\sum_{i=1}^N c_i^2}} \text{ and } Pr[X\leq \E[X]-t]\leq e^{-{\frac{t^2}{2\sum_{i=1}^N c_i^2}}}.\]
\end{theorem}

In some of our proofs we make use of the following construction of a new graph distribution from two graph distributions.

\begin{definition}\label{def:convexcombination}
    Let $C,D$ be two distributions on graphs on $n$ vertices. For any $0\leq\alpha\leq 1$, the \emph{convex combination} $\alpha\cdot C+(1-\alpha)\cdot D$ is the distribution given by first flipping a biased coin with probability of coming up heads being $\alpha$, and then sampling from $C$ if the coin shows heads, and sampling from $D$ otherwise.
\end{definition}
\begin{lemma}\label{lem:convexcombinations}
    Let $C,D$ be two distributions of graphs on $n$ vertices, such that for each edge $e$, the marginal probability of $e$ being present is the same in both distributions, i.e.,  $Pr_C[X_e]=Pr_D[X_e]$. Then, for any $\alpha$, if both $C$ and $D$ are pairwise, matching, or edge-subgraph independent, then respectively so is $\alpha\cdot C + (1-\alpha)\cdot D$.
\end{lemma}
The proof can be found in \Cref{app:convexcombproof}. Note that the condition $Pr_C[X_e]=Pr_D[X_e]$ cannot be omitted. Without this condition, the convex combination is again a probability distribution, but independence is lost. To see this, simply note that every one-point distribution on a graph on $n$ vertices (i.e., $\Pr[G = G_0] =1$ for a fixed graph $G_0$) is trivially a coloring model and thus also satisfies all other independence conditions. Moreover, \emph{every} distribution on graphs of $n$ vertices can be written as a convex combination of one-point distributions, which would lead to a contradiction if the lemma  were true without the condition $Pr_C[X_e]=Pr_D[X_e]$.
Also note that the lemma does not hold for subgraph independence nor for coloring models; a counterexample will come up in the proof of \Cref{lem:sub_esub_inclusion}.

\section{A Strict Hierarchy}\label{sec:hierarchy}
This section is dedicated to proving \Cref{thm:hierarchy}:
\hierarchy*

We will prove each strict inclusion as its own lemma, from left to right. The first among these lemmas shows that each coloring model is subgraph independent but some subgraph independent distributions are not coloring models, which was known from prior work~\cite{aaronson1989algebraic, burton19931, day2020oneindependent, holroyd2016finitely}, but we include the proof here for completeness.
\begin{lemma}
    \label{lem:col_sub_inclusion}
    $\calG_{\text{col}}\subsetneq \calG_{\text{sub}}$, i.e., every coloring model is subgraph independent, but there are some subgraph independent distributions that are not coloring models.
\end{lemma}
\begin{proof}
    We first prove $\calG_{\text{col}}\subseteq \calG_{\text{sub}}$. Consider some graph distribution $D\in \calG_{\text{col}}$ and consider any two disjoint subsets of vertices $V,W$. The resulting graph within $V$ depends only on elementary experiments on $V$, and the resulting graph within $W$ depends only on elementary experiments on $W$. Since the elementary experiments are mutually independent, we have independence of the resulting graphs within $V$ and $W$. Thus, $D\in \calG_{\text{sub}}$.

    Next, we prove $\calG_{\text{col}}\not=\calG_{\text{sub}}$.\footnote{Note that this inequality actually follows from the bounds in \Cref{thm:main}. We nonetheless provide an example of a subgraph independent distribution that is not a coloring model for additional intuition.} Consider the graph distribution $\mathcal{T}$ which uniformly picks one of the three graphs on $3$ vertices with $1$ edge. Since there are only three vertices, $\mathcal{T}$ is clearly subgraph independent. Towards a contradiction, we assume that it is also a coloring model. Consider any vertex $v$ and its anti-clockwise neighbor $u$ and clockwise neighbor $w$. The outcomes of the elementary experiment at $v$ can be decomposed into the outcomes which allow the edge $\{u,v\}$ (but not $\{w,v\}$) to be present, and those which allow the edge $\{w,v\}$ (but not $\{u,v\}$) to be present. There can be no other outcomes, since otherwise the resulting graph would have zero or more than one edges with non-zero probability, which does not happen in $\mathcal{T}$. We now consider the case when the outcome of the experiment at each vertex falls into the part which allows the edge to the clockwise neighbor to be present (and does not allow the edge to the anti-clockwise neighbor to be present). This case must happen with non-zero probability, and leads to a graph with no edges. This contradicts the definition of $\mathcal{T}$, and we conclude that $\mathcal{T}$ is not a coloring model. Thus, $\calG_{\text{col}}\not=\calG_{\text{sub}}$.

    Note that $\mathcal{T}$ can also occur as an induced subgraph in larger subgraph independent graph distributions, and thus the inclusion is also strict when considering only graph distributions on some fixed number $n>2$ of vertices. 
\end{proof}

\begin{lemma}
    \label{lem:sub_esub_inclusion}
    $\calG_{\text{sub}}\subsetneq \calG_{\text{esub}}$, i.e., every subgraph independent distribution is edge-subgraph independent, but there are some edge-subgraph independent distributions that are not subgraph independent.
\end{lemma}
\begin{proof}
     Clearly $\calG_{\text{sub}}\subseteq \calG_{\text{esub}}$ as subgraph independence implies edge-subgraph independence by definition, since a single edge is also a subgraph.

     To prove $\calG_{\text{sub}}\not= \calG_{\text{esub}}$, we consider the following graph distribution $CC(n)$ on $n\geq 6$ vertices: with probability $1/2$, the distribution returns a graph drawn from the \erdos-\renyi distribution $G_{n,1/2}$. Otherwise, the distribution picks a uniformly random decomposition of the vertex set $[n]$ into two sets $A$ and $B$, and returns the graph consisting of the union of cliques on $A$ and $B$. We first show that $CC(n)$ is edge-subgraph independent. In both the \erdos-\renyi regime as well as the two-cliques regime, for any edge $e$ we have $Pr[X_e]=1/2$. This holds even when we condition on the outcome within any subgraph disjoint from $e$, thus $CC(n)$ is edge-subgraph independent. On the other hand, we show that $CC(n)$ is not subgraph independent. To this end, we decompose the vertex set $[n]$ into the sets $P=\{u,v,w\}$ and $Q=[n]\setminus P$. We consider the two events $\mathcal{P}:=\text{``There are exactly two edges within P''}$ and $\mathcal{Q}:=\text{``Q is a clique''}$. Clearly, we have
    \[Pr[\mathcal{Q}]=\frac{1}{2}\cdot 2^{-(n-4)} + \frac{1}{2}\cdot 2^{-\binom{n-3}{2}}>2^{-\binom{n-3}{2}}.\]
    On the other hand, $\mathcal{P}$ implies that we are in the \erdos-\renyi regime, and thus
    \[Pr[\mathcal{Q}\vert\mathcal{P}]=2^{-\binom{n-3}{2}}.\]
    We conclude that $Pr[\mathcal{Q}]>Pr[\mathcal{Q}\vert\mathcal{P}]$ and thus $\mathcal{P}$ and $\mathcal{Q}$ are not independent, showing that $CC(n)$ is not subgraph independent.
\end{proof}

The graph distribution $CC(n)$ that we built in the proof above is a \emph{convex combination} of its two regimes, as defined in \Cref{def:convexcombination}.

\begin{lemma}
    \label{lem:esub_mat_inclusion}
    $\calG_{\text{esub}}\subsetneq \calG_{\text{mat}}$, i.e., every edge-subgraph independent distribution is matching independent, but there are some matching independent distributions that are not edge-subgraph independent.
\end{lemma}
\begin{proof}
    To prove $\calG_{\text{esub}}\subseteq \calG_{\text{mat}}$, consider some edge-subgraph independent distribution. To prove that it is also matching independent, let $M=\{e_1,\ldots,e_k\}$ be some matching. We show for all $i\leq k$ that $Pr[X_{e_i}|X_{e_1},\ldots,X_{e_{i-1}}]=Pr[X_{e_i}]$. This implies \[Pr[X_{e_1} \textrm{ and } X_{e_2} \textrm{ and } \ldots \textrm{ and } X_{e_k}]=\prod_{i=1}^k Pr[X_{e_i}],\] and thus implies matching independence. Let $V'$ be the set of endpoints of the edges $e_1,\ldots,e_{i-1}$. Furthermore, let $\mathcal{G}(V')$ be the set of all possible graphs on $V'$. Then, by the law of total probability, we have
    \[Pr[X_{e_i}|X_{e_1},\ldots,X_{e_{i-1}}]=\sum_{G\in\mathcal{G}(V')}Pr[X_{e_i}|G]Pr[G|X_{e_1},\ldots,X_{e_{i-1}}].\]
    Due to edge-subgraph independence, $Pr[X_{e_i}|G]=Pr[X_{e_i}]$, and we thus have
    \[Pr[X_{e_i}|X_{e_1},\ldots,X_{e_{i-1}}]=Pr[X_{e_i}]\sum_{G\in\mathcal{G}(V')}Pr[G|X_{e_1},\ldots,X_{e_{i-1}}]=Pr[X_{e_i}],\]
    proving the desired claim.

    For the second part of the statement, $\calG_{\text{esub}}\not= \calG_{\text{mat}}$, we consider the following graph distribution $SC(n)$ on $n$ vertices for odd $n\geq 5$: every vertex independently and uniformly picks from the two colors red and blue. Since $n$ is odd, one color was picked by an even number of vertices. The vertices with that color form a clique, and no other edges are present.

    We first show that $SC(n)$ is matching independent: every edge individually occurs with probability $1/4$, since first both endpoints need to have the same color, and second this color must be the color picked by an even number of vertices. We prove that a matching of $k$ edges occurs with probability $(1/4)^k$. For the matching to occur, all $2k$ vertices must pick the same color, which happens with probability $(1/2)^{2k-1}$. Second, that color must end up to be the color picked by an even number of vertices, which has probability $1/2$. Altogether, we have probability $(1/2)^{2k-1+1}=(1/4)^k$.
    
    Finally, we show that $SC(n)$ is not edge-subgraph independent. Let $W$ be a set of $n-2$ vertices, and $e$ be the edge between the remaining $2$ vertices. If $W$ is a clique, $e$ cannot be present. Since both $W$ being a clique and $e$ being present have non-zero marginal probabilities, edge-subgraph independence cannot hold. 
\end{proof}

\begin{lemma}\label{lem:mat_pw_inclusion}
    $\calG_{\text{mat}}\subsetneq \calG_{\text{pw}}$, i.e., every matching independent distribution is pairwise independent, but there are some pairwise independent distributions that are not matching independent.
\end{lemma}

To prove this final strict inclusion lemma, we define the following distribution. This distribution will also be useful later to prove a lower bound on $\rho_{pw}$. 
\begin{definition}
    \label{def:CMmodel}
    $CM(n,q)$ for $n$ even is the following distribution over the graphs on $n$ vertices: to sample $G \sim CM(n,q)$, with probability $q$ we sample from the \emph{clique regime}, and otherwise from the \emph{matching regime}. In the clique regime, we pick a vertex $x\in [n]$ uniformly at random and add all edges $e$ with $x \notin e$ to $G$. In the matching regime, we pick a perfect matching on $[n]$ uniformly at random and add its edges to $G$.
\end{definition}
\begin{claim}
    \label{claim:CMmodel}
    There exists $0 \leq q(n) \leq 1$ with
    $$ q(n) = 1 - \Theta \Big( \frac{1}{n^2} \Big) $$
    such that $CM(n,q(n))$ for $n$ even is pairwise independent but not matching independent, and the probability of each edge is 
    $$ p(n) = 1 - \frac{2}{n} - \Theta \Big( \frac{1}{n^2} \Big).$$
\end{claim}
The proof of this claim is straightforward and just requires some calculations, so we defer it to \Cref{app:CModelProof}.

\begin{proof}[Proof of \Cref{lem:mat_pw_inclusion}]
    Clearly $\calG_{\text{mat}}\subseteq \calG_{\text{pw}}$ as matching independence implies pairwise independence by definition, since two vertex-disjoint edges $e,f$ are a matching.

    To see that $\calG_{\text{mat}}\not= \calG_{\text{pw}}$, recall that by \Cref{claim:CMmodel}, there exists some $q(n)$ such that $CM(n,q)$ is pairwise independent but not matching independent.
\end{proof}

\begin{proof}[Proof of \Cref{thm:hierarchy}]
    \Cref{thm:hierarchy} now follows directly from \Cref{lem:col_sub_inclusion,lem:sub_esub_inclusion,lem:esub_mat_inclusion,lem:mat_pw_inclusion}.
\end{proof}

\section{\texorpdfstring{Bounds on $\rho$}{Bounds on Rho}}\label{sec:bounds}
In this section we prove \Cref{thm:main}, showing lower and upper bounds on $\rho$ for our various models of independence.
Note that by the definition of $\rho$ as in \Cref{sec:detailed-results}, any lower bound on $\rho_X$ for some independence condition $X$ also holds for $\rho_Y$ for some weaker independence condition $Y$, i.e., one such that $\calG_{\text{X}}\subseteq \calG_{\text{Y}}$. Conversely, any upper bound on $\rho_Y$ also holds for $\rho_X$.
\subsection{Pairwise Independence}
\begin{lemma}
    \label{lem:pairwise_lower}
    For even $n\geq 4$, we have $\rho_{\text{pw}}(n)\geq 1-\frac{2}{n}-\Theta(\frac{1}{n^2})$.
\end{lemma}
\begin{proof}
Consider the distribution $CM(n,q)$ as defined in \Cref{def:CMmodel} with $q= 1 - \Theta \Big( \frac{1}{n^2}\Big)$ chosen such that $CM(n,q)$ is pairwise independent, which can be done by \Cref{claim:CMmodel}. The marginal edge probability $p$ in $CM(n,q)$ is then $p=1 - \frac{2}{n} - \Theta \Big( \frac{1}{n^2} \Big)$ . Note that the probability of $G \sim CM(n,q)$ being connected is $0$, since in the clique regime there is always an isolated vertex, and in the matching regime $G$ only consists of a perfect matching. Thus, the threshold $\rho_{\text{pw}}$ for non-zero probability of connectivity in the pairwise independence model is at least $p = 1 - \frac{2}{n} - \Theta \Big( \frac{1}{n^2} \Big)$. 
\end{proof}

\begin{lemma}\label{lem:pairwise_upper}
    For any $n$, we have $\rho_{\text{pw}}(n)\leq 1-\frac{2}{n}$.
\end{lemma}
\begin{proof}
If the minimum marginal edge probability $p$ in any graph distribution (not even necessarily pairwise independent) is larger than $1-\frac{2}{n}$, the expected number of edges $\E[\sum_e X_e]$ is larger than $(1-\frac{2}{n})\binom{n}{2}=\binom{n-1}{2}$. Since any disconnected graph contains at most $\binom{n-1}{2}$ edges, the graph must be connected with non-zero probability.
\end{proof}

\subsection{Matching Independence}
We have no general bounds specifically holding for matching independence. Matching independence is a condition that is surprisingly difficult to exploit. We thus state the following corollary, following directly from \Cref{lem:pairwise_upper} and \Cref{lem:subgraph_lower}.

\begin{corollary}
    \label{cor:matching_bounds}
    For any $n\geq 2$, we have $\frac{1}{2}(1-\tan^2\frac{\pi}{2n})\leq \rho_{\text{mat}}(n)\leq 1-\frac{2}{n}$.
\end{corollary}

The lower bound in \Cref{cor:matching_bounds} can be improved slightly when $n$ is divisible by $8$, which in particular shows that $\rho_{mat}(n) \neq \rho_{sub}(n)$ for infinitely many $n$, illustrating again the different behavior of the connectivity thresholds for different independence conditions.
\begin{lemma}
\label{lem:matching_upper_bound_div_8}
For any $n \in \mathbb{N}$ that is divisible by $8$, we have $\rho_{mat}(n) \geq \frac{1}{2}$.
\end{lemma}
\begin{proof}
We show this by giving a matching independent distribution based on a coloring of the vertices. Consider all possible red-blue colorings $c:[n]\rightarrow \{red,blue\}$ of the vertices, and assign to them the weight function $w(c) = a_{red(c)} 2^{-n}$, where $red(c)$ is the number of red vertices in coloring $c$ and
\[a_k = \begin{cases}
    1,         & \text{if } k \equiv 1 \mod 2,\\
    2,              & \text{if } k \equiv 2 \mod 4,\\
    0,                        & \text{if } k \equiv 0 \mod 4.\\
\end{cases}\]
Noting that
$$ a_k = 1^k - \frac{1}{2} i^k - \frac{1}{2} (-i)^k,$$
where $i$ is the imaginary unit, we get that
$$ \sum_{\text{all colorings }c} w(c) = 2^{-n} \sum_{k \in \{0\} \cup [n]} a_k \binom{n}{k} = 2^{-n} \Big( (1 + 1)^n - \frac{1}{2} (1+i)^n - \frac{1}{2} (1-i)^n \Big)$$
by binomial expansion. Note that $ (1+i)^4 = (2i)^2 = -4$ and $ (1-i)^4 = (-2i)^2 = -4$, so when $n$ is divisible by $8$, we have
$$ \sum_{\text{all colorings }c} w(c) = 2^{-n} \Big( 2^n - (-4)^{n/4} \Big) = 1 - 2^{-n/2} < 1.$$
Thus, we can define a probability distribution over the colorings such that we color the vertices according to coloring $c$ with probability $w(c)$, and we leave all vertices uncolored with probability $2^{-n/2}$. We then obtain our graph by adding a clique on all red vertices and a clique on all blue vertices (in the case when the vertices are uncolored, we leave the graph empty). Note that the resulting graph is never connected since $a_n = a_0 = 0$, which means it cannot be the case that all vertices are blue or that all of them are red.

We now show that this distribution is matching independent. That is, for any non-empty matching $M$ on $K_n$, we show that $Pr[M \in G] = 2^{-|M|}$ if $G$ is sampled from the distribution described above. Let $e \in M$ and consider the colorings of the vertices that are not in $e$ which yield the edges in $M \setminus\{e\}$ to be present in $G$. Note that for $f \in M \setminus e$ to be present, we need both vertices in $f$ to have the same color. Thus, there are $2^{n-1-|M|}$ different colorings of the vertices not in $e$ such that all edges in $M\setminus e$ are present. Let $c$ be one such (partial) coloring. If $red(c)$ is odd, then there are two ways to complete the coloring such that $e$ is present (either both of its endpoints are red, or they are both blue) and each of the resulting colorings occurs with probability $2^{-n}$. If $red(c) \equiv 0 \mod 4$, then the only way to complete this coloring so that $e$ is present and the coloring occurs with positive ($2\cdot 2^{-n}$) probability is to color both endpoints of $e$ red. Similarly, if $red(c) \equiv 2 \mod 4$, then the only way to complete this coloring so that $e$ is present and the coloring occurs with positive ($2\cdot 2^{-n}$) probability is to color both endpoints of $e$ blue. Thus,
$$ Pr[M \in G] = 2^{n-1-|M|} \cdot 2 \cdot 2^{-n} = 2^{-|M|},$$
as desired.
\end{proof}

\subsection{Edge-Subgraph Independence}
We have no lower bound specifically for edge-subgraph independence. We inherit this bound from the stricter subgraph independence, i.e., \Cref{lem:subgraph_lower}.
\begin{corollary}
    \label{cor:edge-subgraph_lower}
    For any $n\geq 2$, we have $\rho_{\text{esub}}(n)\geq \frac{1}{2}(1-\tan^2\frac{\pi}{2n})$.
\end{corollary}

\begin{lemma}
    \label{lem:edge-subgraph_upper}
    For any $n$, we have $\rho_{\text{esub}}(n)\leq \frac{3}{4}$.
\end{lemma}
\begin{proof}
    We show that any edge-subgraph independent distribution with minimum marginal edge probability $p>\frac{3}{4}$ is connected with non-zero probability. To achieve this, we use the \lovasz Local Lemma, as stated in \Cref{lem:LLL}.

    We pick the edges of any Hamiltonian path $e_1,\ldots, e_{n-1}$ of $K_n$ and consider their corresponding events $X_{e_1},\ldots,X_{e_{n-1}}$. Since we have edge-subgraph independence, each of these events depends on only at most $d=2$ of the other events (the ones corresponding to the neighboring edges). Furthermore, the probability of each of these events is at least $p>\frac{3}{4}$. By \Cref{lem:LLL}, since $p> \frac{3}{4}=1-\frac{1^1}{2^2}$, with non-zero probability all of the events $X_{e_1},\ldots,X_{e_{n-1}}$ happen, all edges of the Hamiltonian path are present, and thus the graph is connected.
\end{proof}

\subsection{Subgraph Independence}
The lower bound on $\rho_{sub}$ was given in~\cite[Thm. 16]{day2020oneindependent}, but we provide the proof here for completeness.

\begin{lemma}[{\cite[Thm. 16]{day2020oneindependent}}]\label{lem:subgraph_lower}
    For any $n\geq 2$, we have $\rho_{\text{sub}}(n)\geq \frac{1}{2}(1-\tan^2\frac{\pi}{2n})$.
\end{lemma}
Towards a proof of this lemma, we build a subgraph independent distribution on $n\geq 2$ vertices where the marginal probability of every edge occuring equals $p=\frac12 (1- \tan^2\frac\pi{2n})$, and the probability that a graph sampled from the distribution is connected equals $0$.

Let $\mathcal{K}_n$ be the set of all graphs on vertex set $[n]$ that have at most two connected components and where every connected component is a clique. We will construct a probability distribution on $\mathcal{K}_n$ whose support is all such graphs except $K_n$ and that satisfies the condition of subgraph independence.

Let $z=\frac12 + \frac{\iu}2 \tan \frac\pi{2n}$. Note that $z$ is a complex number. For any graph $G\in \mathcal{K}_n$ we define \begin{equation}\label{eq:defsgindep}p_G:=z^a(1-z)^{n-a}+(1-z)^a z^{n-a},\end{equation}
where $a$ denotes the size of the largest clique in $G$.

One can note that this can be rewritten in terms of trigonometric inequalities by writing $z$ and $1-z$ in polar form. We have  $z=\frac{e^{\pi \iu/2n}}{2\cos \pi/2n} $ and $1-z =\frac{e^{-\pi \iu/2n}}{2\cos \pi/2n}$ hence we have
\begin{equation}\label{eq:sgindeppolar}
p_G = \frac{1}{2^n \cos^n \pi/2n} \left( e^{\pi i (a/n-1/2) } + e^{-\pi i (a/n-1/2)} \right) = \frac{\cos \pi(a/n-1/2) }{2^{n-1} \cos^n \pi/2n},
\end{equation}
from which we immediately see that $p_G$ are non-negative real numbers (since $1/2\leq a/n \leq 1$) such that $p_{K_n}=0$ (as $a=n$ implies $a/n-1/2=1/2$ and $\cos \pi/2=0$).

\begin{proposition}\label{prop:pGs}
The values $p_G$ define a probability distribution on $\mathcal{K}_n$ that has edge density $z^2+(1-z)^2 = \frac12 (1 - \tan^2\frac\pi{2n})$, is subgraph independent and disconnected with probability $1$.
\end{proposition}
\begin{proof}
We have already seen that $p_G\geq 0$ for all $G\in \mathcal{K}_n$, and that $p_G$ is $0$ for the only connected graph, i.e., $K_n$. Hence, in order to show that the values define a probability distribution it remains to show that 
$$ \left(\sum_{G\in\mathcal{K}_n} p_G\right)-1  =0.$$

In order to show that $G$ has edge density $z^2+(1-z)^2$, we must show that for all $e\in{\binom{[n]}{2}}$
$$ \left(\sum_{G\ni e} p_G\right)-z^2 - (1-z)^2=0.$$

Moreover, assuming the condition to be a probability distribution is satisfied, subgraph independence is equivalent to saying that for all disjoint vertex sets $U, V \subseteq [n]$ and all graphs $H_1, H_2$ on $U$ and $V$ respectively we have
$$ \left(\sum_{G[U]=H_1\text{ and }G[V]=H_2} p_G\right) - \left(\sum_{G[U]=H_1} p_G\right) \left(\sum_{G[V]=H_2} p_G\right)=0.$$

Since the values $p_G$ are polynomials in $z$, each of these conditions is a polynomial equation in $z$ which must evaluate to $0$. Since we are interested in the case $z=\frac12 + \frac{\iu}2\tan\frac\pi{2n}$, we need to show that they evaluate to $0$ for this value.
In fact, we show that they evaluate to $0$ for all $z$, i.e., that they are all equal to the zero polynomial. In order to prove this, we show that the conditions are satisfied for $z\in [0, 1]$.

Pick any $z\in [0, 1]$. Consider the random graph distribution on $[n]$ attained by coloring each vertex red with probability $z$ and blue with probability $1-z$, where any two vertices are connected by an edge if they have the same color. Now observe that \eqref{eq:defsgindep} defines exactly the probability to attain any graph $G\in\mathcal{K}_n$, so for this value of $z$, the values $p_G$ define a probability distribution with edge density $z^2+(1-z)^2$. Furthermore, since this distribution is a coloring model, it is also subgraph independent. Hence all aforementioned polynomials are zero for this choice of $z$.

As the only polynomial in $z$ with infinitely many zeros is the zero polynomial, these polynomials must evaluate to zero for any complex number $z$, and the statement follows.
\end{proof}

\begin{proof}[Proof of \Cref{lem:subgraph_lower}]
    By \Cref{prop:pGs}, the probabilities defined in \eqref{eq:defsgindep} define a subgraph independent probability distribution on graphs with $n$ vertices with marginal probability of every edge exactly $\frac12(1-\tan^2\frac\pi{2n})$ and zero probability of being connected. Thus, $\rho_{\text{sub}}(n)\geq \frac{1}{2}(1-\tan^2\frac{\pi}{2n})$.
\end{proof}

The matching upper bound on $\rho_{sub}$ can also be found in \cite{day2020oneindependent}, however, we refrain from restating the proof here.

\begin{lemma}
[{\cite[Thm. 16]{day2020oneindependent}}]\label{lem:subgraph_upper}
    For any $n\geq 2$, we have $\rho_{\text{sub}}(n)\leq \frac{1}{2}(1-\tan^2\frac{\pi}{2n})$.
\end{lemma}

\subsection{Coloring Models}
We first consider coloring models with only $2$ colors, since we can find matching lower and upper bounds for that case. Other restrictions of the coloring model could also be interesting to investigate, such as other bounded numbers of colors, or the case where every vertex must pick \emph{uniformly} among its colors.
\subsubsection{Two Colors}
We show that the threshold probability is exactly $1/4$ in this case.

\begin{lemma}
    \label{lem:coloring_two_colors}
    For all $n\geq 3$, we have $\rho_{col,2}(n)=1/4$.
\end{lemma}
\begin{proof}
    We first prove $\rho_{col,2}(n)\geq 1/4$ by defining the following coloring model which is always disconnected. Decompose the vertex set into three non-empty sets $A,B,C$. Every vertex chooses a color uniformly among the colors red and blue. Two vertices in the same set are connected if they are both red. A vertex $a\in A$ is connected to $b\in B$, if $a$ is red and $b$ is blue. The same goes for $b\in B$ and $c\in C$, as well as $c\in C$ and $a\in A$. Clearly, the marginal edge probability of every edge is $1/4$. We now show that every connected component of a graph sampled from this distribution is a subset of the union of at most two of the sets $A,B,C$. To see this, we pick some edge $\{a,b\}$ for $a\in A$ and $b\in B$ to be present, and try to grow its connected component. Since the edge $\{a,b\}$ is present, $a$ must be red and $b$ blue. Now, $a$ can only be connected to other red vertices in $A$, and to blue vertices in $B$. Similarly, $b$ can only be connected to other red vertices in $A$. No red vertex in $A$ or blue vertex in $B$ can be connected to a vertex in $C$, thus the connected component containing the edge $\{a,b\}$ is contained in $A\cup B$. Symmetrically this holds for any pair of sets, and we conclude that the graph must be disconnected, proving the lower bound.
    
Next, we prove $\rho_{col,2}(n)\leq 1/4$ by showing that if every marginal edge probability is strictly larger than $1/4$, we can always find a coloring that connects the graph. We say that a vertex $v$ \emph{covers} vertex $w$, if for both colors at vertex $w$, there exists a color at $v$, such that the edge $\{v,w\}$ is present under this coloring. We now remove vertices one by one, by repeatedly removing a vertex which covers some remaining vertex, until no more such vertices exist. If the graph on the remaining vertices $V'$ can be connected using some coloring, we can add back the removed vertices in reverse order, and connect them to the vertex they cover, thus connecting the whole graph.
We thus only have to show that the graph on $V'$ can be connected. Since in this graph no vertex covers any other, each edge is only present under exactly one of the four possible color combinations of its endpoints.

We pick a vertex $v'\in V'$ which maximizes $\max(Pr[v\text{ is red}],Pr[v\text{ is blue}])$ among all $v\in V'$. We give this vertex $v'$ the color which is more likely, let $p$ be the probability of this color. One can see that for any other vertex $w\in V'\setminus\{v'\}$, there must be a color such that the edge $\{v',w\}$ is present. Otherwise, the color combination making $\{v',w\}$ present would have probability of at most $p\cdot(1-p)$, which is at most $1/4$. Thus, we can connect $V'$ by simply coloring each vertex with the correct color to connect to $v'$. This concludes the proof of the upper bound, and thus the whole lemma.
\end{proof}

\subsubsection{General Coloring Models}
We now consider the full generality of coloring models, with an unbounded or even infinite (although by \Cref{lem:finitely_many_colors} this can w.l.o.g. be excluded) number of outcomes of the local experiment at each vertex. 

\begin{lemma}
    \label{lem:coloring_lower}
    For all $n\geq 3$, we have $\rho_{\text{col}}(n)\geq (n-2)\frac{3n-4-\sqrt{5n^2-16n+12}}{2(n-1)^2}$.
\end{lemma}
Note that for $n=3$, this bound is equal to $1/4$ (which matches the previous subsection, since our construction only uses $2$ colors in this case), and for $n\rightarrow\infty$, it tends to $2-\phi\approx 0.381966$. We remark that Lemma~\ref{lem:coloring_lower} was obtained simultaneously and independently in~\cite[Thm 1.7]{badakhshian2023density} with the exact same construction.

\begin{proof}
We show that for each number of vertices $n\geq 3$, there exists a coloring model with marginal edge probability $p(n)=(n-2)\frac{3n-4-\sqrt{5n^2-16n+12}}{2(n-1)^2}$ which is connected with probability~$0$. 

We define the distribution $CS(n)$ for any $n\geq 3$. A fixed vertex $v$ picks as its local experiment one other vertex $v'\in [n]\setminus\{v\}$ uniformly at random. All other vertices pick the color red with probability $q$ (to be determined later), and blue with probability $1-q$. Then, $v$ connects to all vertices in $V\setminus\{v,v'\}$ that are colored blue. Between the vertices in $V\setminus\{v\}$, an edge is present if both endpoints are colored red. The resulting graph is clearly not connected: every red vertex is only connected to other red vertices, and in the case that all vertices in $V\setminus\{v\}$ are blue, only $n-2$ edges exist.

For every edge $e$ not incident to $v$ we have $Pr[X_e]=q^2$, and for every edge $e'$ incident to $v$ we have $Pr[X_{e'}]=(1-q)\frac{n-2}{n-1}$. We pick $q$ maximizing the minimum of these probabilities. Since $q^2$ is increasing in $q$ for $q\geq 0$, and $(1-q)\frac{n-2}{n-1}$ is decreasing in $q$, the minimum is maximized when $q^2=(1-q)\frac{n-2}{n-1}$. We can thus solve
\[q^2+\frac{n-2}{n-1}q-\frac{n-2}{n-1}=0\]
to get
\begin{align*}
    q&=-\frac{n-2}{2(n-1)}\pm\sqrt{\frac{(n-2)^2}{4(n-1)^2}+\frac{n-2}{n-1}} \\
    q&=\frac{2-n\pm \sqrt{(n-2)^2+4(n-1)(n-2)}}{2(n-1)} \\
    q&=\frac{2-n\pm\sqrt{5n^2-16n+12}}{2(n-1)}.
\end{align*}

Only one of these solutions fulfills $q>0$, namely the one with ``$+$''. Since the marginal edge probability $p(n)$ is equal to $(1-q)\frac{n-2}{n-1}$, we get the claimed bound.
\end{proof}

Let us now consider upper bounds for $\rho_{\text{col}}(n)$. We first state the weaker constant bound holding for all $n$. Note that the proof of Theorem 1.7 in~\cite{badakhshian2023density}, which came out independently and simultaneously, uses essentially the same construction as the one we employ in the lemma below, except that they optimize the calculations to get a bound of $\frac{1}{2} - \frac{1}{4n-6}$.
\begin{lemma}
    \label{lem:coloring_upper_one_half}
    For every $n$, we have $\rho_{\text{col}}(n)\leq \frac{1}{2}$.
\end{lemma}
\begin{proof}
We show that we can find a connected outcome in every coloring model with minimum marginal edge probability $p\geq 0.5$.
    Let the vertices of the graph be $v_0, \dots, v_{n-1}$. We consider all possible  (say $k$ many\footnote{For readability of this proof we assume that the number of outcomes in the random experiment at each vertex is finite, which we can do by \Cref{lem:finitely_many_colors}. The exact same proof strategy also works for (even uncountably) infinite probability spaces.}) colors for $v_0$, and the conditional probabilities of the edges between $v_0$ and each other vertex $v_i$, conditioned on $v_0$'s color. It is convenient to represent these in a table $A$ with $k$ rows and $n-1$ columns, where $A_{ij} = Pr[X_{v_0v_j}| c(v_0) = i]$, where $c(v_0)=i$ means that $v_0$ is colored in color $i$. Note that by the law of total probability, for each column $j$, we have
    $$ \sum_{i=1}^k A_{ij} Pr[c(v_0)=i] \geq p \geq 0.5.$$
    Summing over all rows, this implies
    $$ \sum_{j=1}^{n-1} \sum_{i=1}^k A_{ij} Pr[c(v_0)=i] \geq \frac{n-1}{2} $$
    $$ \sum_{i=1}^k Pr[c(v_0)=i] \sum_{j=1}^{n-1} A_{ij} \geq \frac{n-1}{2},$$
    and since $\sum_{i=1}^k Pr[c(v_0)=i] = 1$, the latter formulation implies that there is at least one row $i$ for which
    $$ \sum_{j=1}^{n-1} A_{ij} \geq \frac{n-1}{2}.$$
    Fix one $i$ for which this holds and let $x$ be the number of entries $j$ on that row with $A_{ij} = 0$, let $y$ be the number of entries $j$ with $0 < A_{ij} \leq \frac{1}{2}$ and let $z$ be the number of entries $j$ with $ A_{ij} > \frac{1}{2}$. Then it must be the case that $z \geq x$, as otherwise we would have
    $$ \sum_{j=1}^{n-1} A_{ij} \leq x \cdot 0 + y \cdot \frac{1}{2} + z \cdot 1 < (x+z)\cdot \frac{1}{2} + y \cdot \frac{1}{2} = \frac{n-1}{2},$$
    contradicting our choice of the row $i$.
    
    To show that the graph is connected with non-zero probability, we exhibit a choice of colors for which every vertex has a path to $v_0$ of length at most $2$. We choose color $i$ for $v_0$. Since we showed above that $z \geq x$, we can pair up each vertex $v_j$ for which $A_{ij} = 0$ with a distinct vertex $v_{j'}$ for which $A_{ij'} > \frac{1}{2}$. We cannot connect $v_j$ to $v_0$ directly given that $v_0$ has color $i$, so we will connect $v_j$ to $v_{j'}$ and $v_{j'}$ to $v_0$. Since $$Pr[X_{v_0v_{j'}}|c(v_0)=i] + Pr[X_{v_{j'}v_j}] > \frac{1}{2}+p \geq 1,$$ there is some choice of colors for $v_{j'}$ and $v_{j}$ such that both edges $v_0v_{j'}$ and $v_{j'}v_j$ are present. We fix these colors for $v_j$ and $v_{j'}$. Repeating this for each vertex with a $0$ entry in our table, we are left only with some vertices $v_h$ with $A_{ih} > 0$ (namely the ones that were not paired up with any vertex with a $0$ entry, and the ones for which the entry in the table is positive but at most $\frac{1}{2}$). For each such $v_h$, there is a choice of color such that the edge $v_0v_h$ is present, so we pick that color. This gives us a spanning tree rooted at $v_0$, where all vertices $v_h$ with $A_{ih}>0$ are neighbors of $v_0$, and all the others are neighbors of neighbors of $v_0$ (see \Cref{fig:upper_bound_half} for an example), completing the proof. \qedhere

    \begin{figure}[tb]
        \centering
        \includegraphics[]{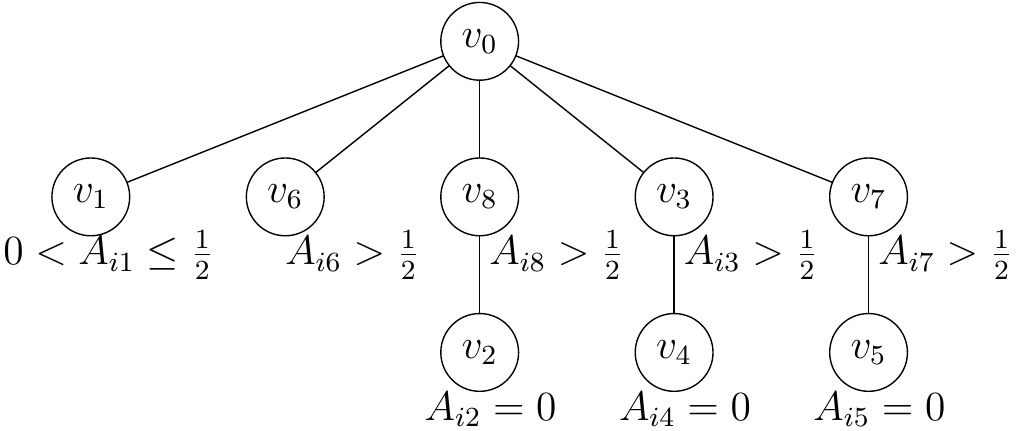}
        \caption{In this example, there are three vertices $v_j$ with $A_{ij}=0$, four vertices $v_j$ with $A_{ij} > \frac{1}{2}$, and one vertex $v_j$ with $0 < A_{ij} \leq \frac{1}{2}$.}
        \label{fig:upper_bound_half}
    \end{figure}
\end{proof}

Next, we present the stronger upper bound holding for large enough $n$. This bound tends to $2-\phi$ for $n\rightarrow\infty$ and thus matches the lower bound for large enough $n$, as conjectured in~\cite{badakhshian2023density}.

\begin{theorem}\label{thm:coloring_upper_n_infinity}
For any $\epsilon>0$, there exists an $n_0(\epsilon)$, such that for any coloring model distribution on graphs on $n\geq n_0(\epsilon)$ vertices with minimum marginal edge probability at least $p' = 2-\phi+\epsilon$, the graph is connected with non-zero probability.
\end{theorem}

To prove this theorem, we will need the following notation and setup. We write $p$ for $2-\phi$.

\begin{observation}
\label{lemma:nice-identities}
For $p= 2 - \phi$, we have the nice identities $ \frac{1}{p} - 2 = 1-p$ and $(1-p)^2=p$.
\end{observation}

To begin, we color each vertex independently with a random color according to its own distribution.\footnote{We once again assume that we have a finite number of colors. Our proof strategy also works with infinite probability spaces.} For this fixed coloring, we write $N(v,c)$ for the set of neighbors of $v$, if $v$ is recolored to the color $c$.

\begin{lemma}\label{lem:condition}
In a random coloring, with high probability it holds that $\E_{c}[|N(v,c)|]>(p+\epsilon/2)n$ for all $v$ simultaneously.
\end{lemma}
\begin{proof}
    Fix a vertex $v$, and for each $w \neq v$, let $Y_w$ be the random variable over the color of $w$ that denotes the probability that the edge $vw$ is present when sampling the color of $v$ from its distribution. Note that $\E[Y_w] \geq p'$. Then 
    \[Y:= \E_{c}[|N(v,c)|] = \sum_{w\neq v} Y_w.\]
    We can then apply Azuma's inequality with $Y_w$ as coordinates, noting that the effect of each $Y_w$ on $Y$ is at most $1$ since $0\leq Y_w \leq 1$. We have
    \[ \E[Y] = \sum_{w \neq v} \E[Y_w] \geq (n-1) p', \]
    so letting $t := \frac{\varepsilon}{4}n$, using \Cref{thm:azuma}, we get the following bounds for large enough $n$:
    $$ Pr[Y \leq (p+\epsilon/2)n] \leq Pr[Y \leq (p'-\varepsilon/2)n] \leq Pr[Y \leq \E[Y] - t] \leq e^{- \frac{t^2}{2(n-1)}}\leq e^{- \varepsilon^2 n / 32}.$$
    
    By a union bound over all vertices $v$, the event that $\E_c[|N(v,c)|]\leq (p+\eps/2)n$ for any vertex $v$ has probability at most $n e^{-\eps^2n/32}$, which tends to $0$ for $n\rightarrow\infty$, proving the lemma.
\end{proof}
Thus, from now on, we condition on the high probability event of \Cref{lem:condition}. This allows us to argue that we can recolor any vertex such that it has a large neighborhood. Furthermore, using \Cref{lem:markovderivate}, we can also argue that for every vertex, relatively large neighborhoods must exist with a somewhat large probability. This allows us to use union bounds to show that large neighborhoods must exist at the same time as some fixed edges.

Our general strategy for showing the graph is connected with positive probability will be to mostly rely on the random coloring of all vertices that we start with, but recolor some vertices as necessary, using some large neighborhoods to connect vertices to.

Since we will recolor some vertices, the true neighborhoods in the final graph may be different (possibly also smaller) than in the random coloring for which we get these bounds. We use the error term $C:=C'(\epsilon)\log_2 n$ for an upper bound on the number of vertices that we recolor and aim to connect to (we may recolor more vertices, but we do not argue that we connect to these vertices, thus recoloring them does not matter). If we can guarantee a desired intersection of neighborhoods (or of a neighborhood with a fixed set) to contain at least $C+1$ vertices in the random coloring, we know that in the actual coloring at the end the intersection is non-empty.

Let $(r,c_r)$ be a vertex-color pair which maximizes $|N(r,c_r)|$. We recolor $r$ to $c_r$. If every vertex outside of $N(r,c_r)$ can be recolored to some color such that it has an edge to a vertex in $N(r,c_r)$, the graph can be connected. Otherwise, we pick a vertex-color pair $(b,c_b)$ which maximizes $|N(b,c_b)|$ among all vertex-color pairs for which $b\not\in N(r,c_r)\cup\{r\}$, and $|N(r,c_r)\cap N(b,c_b)|\leq C$.
We recolor $b$ to $c_b$.

In the following, we use the shorthands
\[R:=N(r,c_r),\; B:=N(b,c_b),\; \alpha:=|R|/n, \text{ and }\beta:=|B|/n.\]
The vertices $r$ and $b$, along with their respective neighborhoods $R$ and $B$, will be central to our argument, and we will find various ways to connect all other vertices to those sets.

\begin{lemma}
    $p<\alpha< 1-p$, and $p<\beta\leq\alpha$.
\end{lemma}
\begin{proof}
    We prove the four bounds independently:
    \begin{itemize}
        \item[$\alpha>p$:] Since $\E_{c}[|N(v,c)|]>pn$ for all $v$, clearly $\alpha n=|N(r,c_r)|$, the largest of all such neighborhoods, must be larger than $pn$ as well.
        \item[$\alpha<1-p$:] If we had $\alpha\geq 1-p$, a maximum size neighborhood (which has size at least $pn$) of each other vertex would need to intersect $R$. We assumed this is not the case, thus $\alpha<1-p$.
        \item[$\beta>p$:] Consider some vertex $v$ for which $(N(v,c_v)\cup \{v\}) \cap R=\varnothing$ for every color $v$. All vertex-color pairs $(v,c')$ were considered when maximizing to find $(b,c_b)$. Since $\E_c[|N(v,c)|]>pn$, we thus have $\beta n=|N(b,c_b)|>pn$, too.
        \item[$\beta\leq\alpha$:] This follows from the fact that $(b,c_b)$ is picked from a constrained set of vertex-color pairs, while $(r,c_r)$ is picked among all possible pairs. \qedhere
    \end{itemize}
\end{proof}

Note that if we have $R\cap B\not=\varnothing$, we can easily connect the graph: if we recolor each vertex $v\not\in R\cup B\cup\{r,b\}$ to some color $c$ such that $|N(v,c)|>pn$, $N(v,c)$ must intersect $R\cup B$. This follows from $|R\cup B| = |R|+|B|-|R\cap B|\geq \alpha n + \beta n-C\geq 2pn-C$, and $(2pn-C)+pn>n$. We thus assume in the following that $R\cap B=\varnothing$.

\begin{lemma}
    $\beta\leq\min(1-\alpha,1/2)$.
\end{lemma}
\begin{proof}
    This follows directly from $R\cap B=\varnothing$ and thus $\alpha+\beta\leq 1$.
\end{proof}

To connect the graph in the remaining cases, we need to introduce some more definitions.
Intuitively, we say that any vertex $v$ is \emph{obligate red}, if it cannot robustly be connected to the blue set $B$. More formally, we say that a vertex $v\not\in B\cup\{b\}$ is obligate red, if there exists no color $c$, such that $|N(v,c)\cap B|> C$. Similarly, we say that a vertex $v\not\in R\cup\{r\}$ is \emph{obligate blue}, if there exists no color $c$ such that $|N(v,c)\cap R|>C$. Let $OR$ be the set of obligate red vertices, and $OB$ the set of obligate blue vertices.

Note that no vertex can be both obligate blue and obligate red. When a vertex is neither obligate blue nor obligate red, we say it is \emph{non-obligate}.

We start with giving some guarantees on the sizes of neighborhoods of obligate blue and obligate red vertices. 
\begin{lemma}
    \label{lemma:neighborhood_size_guarantees}
    Let $u \in OB, v \in OR$. Then, over the distribution of the color of $u$ $(v)$, the following hold for the neighborhood of $u$ $(v)$.
    \begin{enumerate}
        \item $Pr[|N(u,c)| \geq \frac{p+p\beta -\beta}{p}n + \eps n /2] \geq 1-p$.
        \item $Pr[|N(v,c)| \geq \frac{p+p\alpha -\alpha}{p}n + \eps n /2] \geq 1-p$.
        \item $Pr[|N(u,c)| \geq (1-p)(1-\beta)n + \varepsilon n / 2] \geq 1-\sqrt{p}$.
        \item $Pr[|N(v,c)| \geq (1-p)(1-\alpha)n + \varepsilon n / 2] \geq 1-\sqrt{p}$.
    \end{enumerate}
\end{lemma}
\begin{proof}[Proof (sketch)]
We can upper bound the size of a neighborhood of any vertex by $\alpha n$. Furthermore, we can upper bound the size of a neighborhood of any obligate blue vertex $u$ by $\beta n$. The lemma follows from applying \Cref{lem:markovderivate} and simplifying. The full calculations are found in \Cref{app:calculations}.
\end{proof}

\begin{corollary}\label{cor:obligateblueConnectsToB}
Let $S$ be a vertex set of size $|S|\leq C$. Let $u\in OB$. Then, 
$$Pr_c[N(u,c)\cap (B\setminus S)\neq \varnothing]\geq1-p.$$
\end{corollary}
\begin{proof}
    By \Cref{lemma:neighborhood_size_guarantees}, part 1, we have
    $$Pr[|N(u,c)| \geq \frac{p+p \beta - \beta}{p}n + \eps n /2] \geq 1-p.$$
    We wish to show that when $|N(u,c)|\geq \frac{p+p\beta-\beta}{p}n+\eps n/2$, we have that $N(u,c)$ and $B\setminus S$ intersect. Since $|N(u,c) \cap R| \leq C$ and $(B\setminus S)\cap R=\varnothing$, we have that $|N(u,c)|+|B\setminus S|$ can be at most $n-|R|+C$ if the sets do not intersect. It thus suffices to show that
    $$ \frac{p+p \beta - \beta}{p}n + \eps n/ 2  + (\beta n - C) \overset{!}{>} n - \alpha n + C, $$
    or equivalently,
    $$ \frac{p+p\beta -\beta}{p} + \eps/2 + \alpha  + \beta  - 2C/n - 1 \overset{!}{>} 0.$$
    First, we use that $\epsilon/4>2C/n$ for $n$ large enough.
    $$ \frac{p+p\beta -\beta}{p} + \eps/2 + \alpha + \beta - 2C/n - 1 > \frac{p + p\beta - \beta}{p} + \eps/4 + \alpha + \beta - 1 $$
    $$ = (2-\frac{1}{p})\beta + \eps/4 + \alpha = (p-1)\beta + \alpha + \eps/4 > (\alpha - \beta) + \eps/4 \geq \eps /4> 0,$$
    since $\alpha \geq \beta$. Note that we used \Cref{lemma:nice-identities} here.
\end{proof}

The symmetric statement for obligate red vertices also follows from \Cref{lemma:neighborhood_size_guarantees}:

\begin{corollary}\label{cor:obligateredConnectsToR}
Let $S$ be a vertex set of size $|S|\leq C$. Let $v\in OR$. Then, 
$$Pr_c[N(v,c)\cap (R\setminus S)\neq \varnothing]\geq 1-p.$$
\end{corollary}
\begin{proof}[Proof (sketch)]
    The proof works the same way as the proof of \Cref{cor:obligateblueConnectsToB}, except that we use part 2 of \Cref{lemma:neighborhood_size_guarantees}, and we use that $\beta \geq \alpha(1-p)$ since $\beta \geq p$ and $\alpha \leq 1-p$.
\end{proof}

We will now prove \Cref{thm:coloring_upper_n_infinity} by case distinction on the number of obligate red and obligate blue vertices.

\begin{lemma}\label{lem:noobligates}
    If there are either no obligate red or no obligate blue vertices, the graph can be connected.
\end{lemma}
\begin{proof}
    If there are no obligate red vertices, we recolor every vertex $v\not\in B\cup\{b\}$ to a color $c$ such that $N(v,c)$ intersects $B$. Every vertex is thus connected to $b$ either directly or through a vertex in $B$, proving that the graph is connected.
    Symmetrically, if there are no obligate blue vertices, all vertices can be connected to $r $ or $R$.
\end{proof}

\begin{lemma}\label{lem:manyobligates}
    If $|OR|>C$  or $|OB|>C$, the graph can be connected.
\end{lemma}
\begin{proof}
   
    \begin{figure}[tb]
        \centering
        \includegraphics{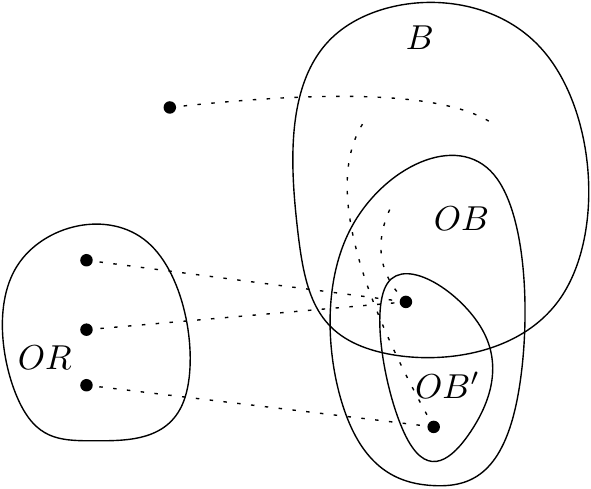}
        \caption{How to connect the graph when $|OB|>C$.}
        \label{fig:more_than_log}
    \end{figure}

    Suppose without loss of generality that $|OB|>|OR|$ (the other case is symmetric). Let $OB'\subset OB\setminus \{b\}$ be a subset of $C$ obligate blue vertices. We will recolor all vertices in $OB'$ as well as all non-obligate vertices such that they connect into $B\setminus OB'$. We will prove that each vertex in $OB'$ has a color that allows it to connect to $B\setminus OB'$ as well as a constant proportion of the remaining vertices in $OR$. Since $C$ is logarithmic in $n$, we are able to cover all vertices in $OR$ in this way, and thus connect the graph. A sketch of this can be seen in \Cref{fig:more_than_log}. 
    
    For any $u \in OB'$ and $v \in OR$, if we let each of them pick a color independently at random from their respective distributions, we have that $Pr[uv \in E(G)] \geq p' = p + \eps$. Furthermore, by \Cref{cor:obligateblueConnectsToB}, $Pr[N(u,c)\cap(B\setminus OB')\neq \varnothing]\geq (1-p)$.
    Thus, letting $Y_{u,v}$ be the indicator random variable that is $1$ if and only if $uv \in E(G)$ and $N(u,c) \cap (B\setminus OB') \neq \varnothing$, we have
    $$ Pr[Y_{u,v}=1] \geq p' + (1 - p) - 1 \geq \eps.$$
    Set $Y_u = \sum_{v \in OR}Y_{u,v}$ and note that if $N(u,c)$ and $B
    \setminus OB'$ do not intersect, then $Y_u=0$, and otherwise $Y_u$ counts the neighbors of $u$ in $OR$ (according to the new coloring of $OR$, not the one we started with). Choosing the colors of $u$ and all $v \in OR$ independently at random from their respective distributions, by linearity of expectation we get
    $$ \E[Y_u] = \sum_{v\in OR} \E[Y_{u,v}] \geq \eps |OR|. $$
    This implies that there is a color $c$ for $u$ and a choice of colors for the vertices in $OR$ such that $Y_u \geq \eps |OR|$, that is, such that $N(u,c) \cap (B\setminus OB') \neq \varnothing$ and $u$ has at least $\eps |OR|$ neighbors in $OR$ according to the new coloring of $OR$. We color $u$ in $c$ and fix the colors assigned to these $\eps|OR|$ neighbors of $u$ by the new coloring. Thus, all these at least $\eps |OR|$ vertices in $OR$ become neighbors of $u$, and $u$ has an edge to some vertex in $B\setminus OB'$.
    
    By using one vertex from $OB'$, we thus connected an $\eps$-fraction of $OR$. We repeat this process until no more vertices remain in $OR$. Note that we can always match at least one vertex, which we must take into account as soon as $\eps |OR|$ becomes smaller than $1$. We thus use at most
    $$1+ \eps^{-1} + \log_{\frac{1}{1-\eps}}(|OR|)\leq 1+\eps^{-1} + \frac{\log_2(n)}{\log_2(\frac{1}{1-\eps})}\leq C$$
    vertices from $OB'$, picking $C'(\eps)$ large enough with respect to $\eps$.
    
    The case of $|OR| > |OB|$ works in a symmetrical way, using \Cref{cor:obligateredConnectsToR}.
\end{proof}

We thus assume in the following that the numbers of obligate red and obligate blue vertices are more than~$0$ and at most $C$.

\begin{lemma}\label{lem:matching}
    If there is an injective mapping from $OR$ to $OB\setminus\{b\}$, or an injective mapping from $OB$ to $OR\setminus\{r\}$, the graph can be connected.
\end{lemma}
\begin{proof}
    We first assume an injective mapping $f$ from $OR$ to $OB\setminus\{b\}$ exists. The opposite case works symmetrically.
    In this case we aim to connect every vertex to $b$, by connecting each obligate red vertex $v$ to the vertex $f(v)$. These obligate blue vertices, as well as the non-obligate vertices outside $B\setminus\{b\}$ are then connected to some vertex in $B\setminus OB$. Since these vertices are connected to $b$, the whole graph is connected. This is shown in \Cref{fig:matchingred}. Note that we recolor the vertices in $OB$ and outside $B$, but not those in $B\setminus OB$.

    \begin{figure}[tb]
        \centering
        \includegraphics{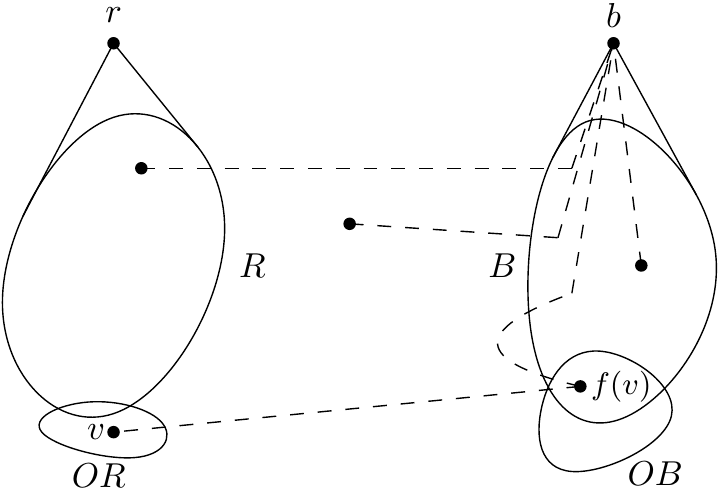}
        \caption{How to connect the graph given an injective mapping $f$ from $OR$ to $OB\setminus\{b\}$.}
        \label{fig:matchingred}
    \end{figure}

    As in the proof of \Cref{lem:manyobligates}, we let $v\in OR$ and $f(v)\in OB$ pick random colors independently. Then, by \Cref{cor:obligateblueConnectsToB} and since $Pr[vf(v)\in E(G)]\geq p'$, there is again a non-zero probability of both $v$ connecting to $f(v)$ and $f(v)$ connecting to $B\setminus OB$. Thus, we can pick such colors for each pair of vertices $(v,f(v))$, for $v\in OR$.
\end{proof}

We can enhance the previous lemma by considering non-obligate vertices which also have a large probability (at least $1-p$) to robustly connect to either of the two sets $R$ and $B$ (as obligate vertices are guaranteed to, by \Cref{cor:obligateblueConnectsToB} and \Cref{cor:obligateredConnectsToR}, respectively):

\begin{lemma}\label{lem:vertexActingAsObligate}
    If there exists a non-obligate vertex $v\not\in\{r,b\}$ such that $Pr_c[|N(v,c)\cap S|>C]\geq 1-p$ for $S=R$ or $S=B$, the graph can be connected.
\end{lemma}
\begin{proof}
    If the conditions of \Cref{lem:manyobligates} or \Cref{lem:matching} are fulfilled, we can connect the graph. Otherwise, $|OR|=|OB|\leq C$ and $r\in OR$ and $b\in OB$. Assume there is a vertex $v$ such that $Pr_c[|N(v,c)\cap B|>C]\geq 1-p$ (the case $S=R$ works symmetrically). We can make a bijective mapping from $OR$ to $(OB\cup\{v\})\setminus\{b\}$. Since we have that $Pr_c[N(v,c)\cap (B\setminus (OB\cup \{v\}\setminus \{b\}))\neq \varnothing]\geq 1-p$, the proof of \Cref{lem:matching} also works for this mapping.
\end{proof}

Furthermore, if there exists a vertex that can be used to join the red and blue sets, we can also connect the graph:

\begin{lemma}\label{lem:purplevertex}
    If there exists a vertex $v\not\in\{r,b\}$ and a color $c$, such that $N(v,c)\cap R\not=\varnothing$ and $N(v,c)\cap B\not=\varnothing$, we can connect the graph.
\end{lemma}
\begin{proof}
    We give the vertex $v$ the color $c$. Every vertex $w\not\in R\cup B\cup\{r,b,v\}$ is recolored to some color $c_w$ such that $|N(w,c_w)|>pn$. Since $p>1-\alpha-\beta$, $w$ is then connected to some vertex in $R$ or $B$. After doing this for all $w$, the graph is connected.
\end{proof}

Finally, we will consider the cases which were not covered by the previous lemmata, for which we need the following lower bounds regarding non-obligate vertices.

\begin{lemma}\label{lem:bprimeu}
    For any non-obligate vertex $v\not\in\{r,b\}$, we define
    \begin{align*}
        p^r_v&:=\max_{c:|N(v,c)\cap B|\leq C}|N(v,c)|, \\
        p^b_v&:=\max_{c:|N(v,c)\cap R|\leq C}|N(v,c)|.
    \end{align*}
    Assuming there exists no vertex fulfilling the conditions of either \Cref{lem:vertexActingAsObligate} or \Cref{lem:purplevertex}, we have $p^r_v > \frac{\beta p - \beta + p}{p}n$, and $p^b_v > \frac{\alpha p - \alpha + p}{p}n$.
\end{lemma}
\begin{proof}
    Let $v$ be an arbitrary non-obligate vertex, and let $q_v:=Pr_{c_v}[|N(v,c_v)\cap R|>C]$ and $z_v:=Pr_{c_v}[|N(v,c_v)\cap B|>C]$. Since $v$ does not fulfill the conditions of \Cref{lem:vertexActingAsObligate}, we have $q_v,z_v< 1-p$. We now look at the expected size of $N(v,c_v)$ when picking a random color $c_v$. First, we have $pn<\E_{c_v}[|N(v,c_v)|]$ by \Cref{lem:condition}.
    To obtain the bound on $p^b_v$, we can use that we also have 
    $$\E_{c_v}[|N(v,c_v)|]\leq q_v\alpha n + (1-q_v)p^b_v.$$
    This holds since every neighborhood has size at most $\alpha n$, and by definition of $p^b_v$, any neighborhood of $v$ intersecting $R$ in at most $C$ vertices has size at most $p^b_v$. Combining these bounds, we get $pn<q_v\alpha n+(1-q_v)p^b_v$, which implies
    \begin{gather*}
        p^b_v > \frac{p-q_v\alpha}{1-q_v}n=\left(\alpha-\frac{\alpha-p}{1-q_v}\right)n\geq \left(\alpha - \frac{\alpha-p}{p}\right)n=\frac{\alpha p -\alpha + p}{p}n,
    \end{gather*}
    as desired.

    To obtain the bound on $p^r_v$, we can use the bound $\E_{c_v}[|N(v,c_v)|]\leq z_v\beta n + (1-z_v)p^r_v$. Here, the part $z_v\beta$ follows from the fact that any neighborhood intersecting $B$ cannot intersect $R$ (since we assumed no vertex as in \Cref{lem:purplevertex}), and was thus considered when picking the pair $(b,c_b)$. Thus, we have $pn<z_v\beta n+ (1-z_v)p^r_v$, implying
    \begin{gather*}
        p^r_v > \frac{p - z_v\beta}{1-z_v}n=\left(\beta - \frac{\beta - p}{1-z_v}\right)n\geq \left(\beta - \frac{\beta - p}{p}\right)n=\frac{\beta p - \beta + p}{p}n,
    \end{gather*}
    finishing the proof.
\end{proof}

Finally, we cover the only remaining case.

\begin{lemma}\label{lem:everythingelse}
    If $|OB|=|OR|$, and there exists no vertex as in \Cref{lem:vertexActingAsObligate} or \Cref{lem:purplevertex}, we can connect the graph.
\end{lemma}
\begin{proof}
    We consider the potential edges between $OR$ and $OB$. Each such edge $r'b'$ can be oriented in some direction, by considering the following two probabilities: 
    \begin{gather*}
    p_{r' \rightarrow b'}:=Pr_{c_{r'}}[\exists c_{b'} \text{ such that $r'b'\in E(G)$ in coloring $c(r')=c_{r'},c(b')=c_{b'}$}] \\
    p_{b' \rightarrow r'}:=Pr_{c_{b'}}[\exists c_{r'} \text{ such that $r'b'\in E(G)$ in coloring $c(r')=c_{r'},c(b')=c_{b'}$}]
    \end{gather*}
    Note that $p_{r' \rightarrow b'}\cdot p_{b' \rightarrow r'}\geq Pr[r'b'\in E(G)]\geq p'$, and thus $\max (p_{r' \rightarrow b'},p_{b' \rightarrow r'})\geq \sqrt{p'}$. We now direct the potential edge $r'b'$ from $r'$ to $b'$ if $p_{r'\rightarrow b'}>p_{b' \rightarrow r'}$, and from $b'$ to $r'$ otherwise.
    We pick an arbitrary perfect matching $M$ among the potential edges between $OR$ and $OB$, using $|OR|=|OB|$. 

    \begin{figure}[tb]
        \begin{subfigure}{0.49\textwidth}
        \includegraphics[keepaspectratio,width=\textwidth]{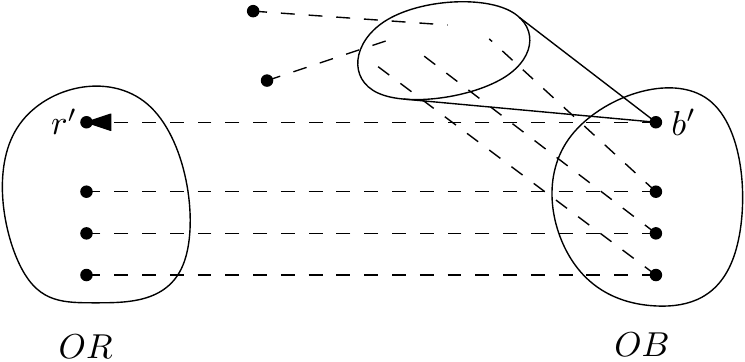}
        \caption{Case 1: At least one edge is oriented towards~$OR$.}
        \label{subfig:case1}
        \end{subfigure}
        \begin{subfigure}{0.49\textwidth}
        \includegraphics[keepaspectratio,width=\textwidth]{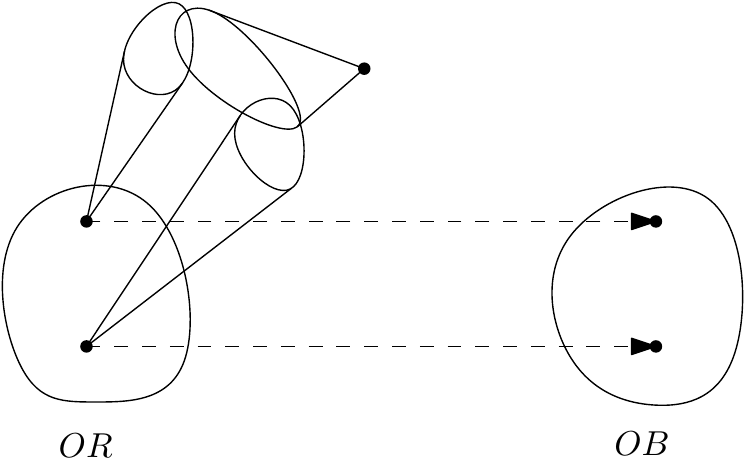}
        \caption{Case 2: All edges are oriented towards $OB$.}
        \label{subfig:case2}
        \end{subfigure}
        \caption{Connecting the graph in \Cref{lem:everythingelse}.}
        \label{fig:cases}
    \end{figure}
    
    \textbf{Case 1:} There exists an arc $(b',r')\in M$ directed towards $r'\in OR$.
    In this case we will connect the graph as shown in \Cref{subfig:case1}. Every obligate red vertex is connected to its matching obligate blue vertex. The vertex $b'$ must have a large neighborhood size that can be guaranteed with probability $\geq 1-\sqrt{p}$, so that it can simultaneously be connected to $r'$. Every obligate blue and every non-obligate vertex connects to this neighborhood.

    By \Cref{lemma:neighborhood_size_guarantees}, part 3, we have
    $$Pr[|N(b',c)|\geq (1-p)(1-\beta)n+\epsilon n/2]\geq 1-\sqrt{p}.$$
    Thus, with positive probability over the choice of color for $b'$, the size of the neighborhood of $b'$ is at least $(1-p)(1-\beta)n + \eps n/2$ and there is a choice of color for $r'$ such that $r'b'$ is an edge, since $1- \sqrt{p} + \sqrt{p'}>0$.

    Similarly, for all other obligate blue vertices $u$, by \Cref{lemma:neighborhood_size_guarantees}, part 1, we have
    $$Pr[|N(u,c)|\geq \frac{p+p\beta -\beta}{p}n+\epsilon n/2]\geq 1-p,$$
    so with positive probability there is an edge between $u$ and its matching vertex in $OR$ and $u$'s neighborhood has size at least $\frac{p+p\beta -\beta}{p}n+\epsilon n/2$.

    We need to prove that these guaranteed sizes of neighborhoods must intersect outside of $OB\cup OR$. Let $S:=(1-p)(1-\beta)n+\epsilon n/2$ and $T:=\frac{p+p\beta - \beta}{p}n+\epsilon n/2$. Since both neighborhoods can intersect $R$ in at most $C$ vertices, we need to prove that $(S-3C) + (T-3C) > n - |R|$, i.e.,
    \begin{equation}\label{eqn:compare1}
        S + T + \alpha n - n - 6C \overset{!}{>} 0. 
    \end{equation}

    Additionally, to guarantee that each non-obligate vertex $u$ outside of the neighborhood of $b'$ can connect to that neighborhood, we need to guarantee that $u$ has a neighborhood which intersects that of $b'$ outside $OB\cup OR$. For this, we need to check that $(S-3C) + (p^b_u-3C) > n - |R|$, i.e.,
    \begin{equation}\label{eqn:compare2}
    S + p^b_u + \alpha n- n - 6C \overset{!}{>} 0.
    \end{equation}

    Since our lower bound on $p^b_u$ is smaller than $T$ (as $\alpha\geq \beta$), (\ref{eqn:compare2}) implies (\ref{eqn:compare1}), so it suffices to show (\ref{eqn:compare2}).
    $$ S+p^b_u +\alpha n - n - 6C \geq (1-\beta)(1-p)n + \frac{p + \alpha p - \alpha }{p}n +\alpha n - 6C - n + \eps n / 2$$
    $$ = \left(-\beta - p + \beta p  + 1 + 2\alpha - \frac{\alpha}{p} - 6\frac{C}{n} + \eps /2 \right) n$$
    $$ = \left(1 - p - \alpha(1/p - 2) - \beta(1-p) - 6\frac{C}{n} + \eps / 2\right) n$$
    $$ = \left(1 - p - (\alpha + \beta)(1-p) - 6\frac{C}{n} + \eps / 2 \right)n$$
    $$ \geq \eps n / 2 - 6C,$$
    which for large enough $n$ is larger than $0$. Note that here we used \Cref{lemma:nice-identities}.
    
    \textbf{Case 2:} All edges in $M$ are oriented towards $OB$.
    In this case we will connect the graph as shown in \Cref{subfig:case2}. We give each obligate red vertex a large neighborhood such that we can still guarantee to be able to connect the obligate blue vertices to their matching partner. Then, we give each non-obligate vertex which is not yet connected directly to any of the obligate vertices a large neighborhood not intersecting $B$ (as guaranteed by \Cref{lem:bprimeu}). These neighborhoods must intersect all neighborhoods of the obligate red vertices due to their sizes, thus the graph is connected.
    
    Let $Y:=(1-p)(1-\alpha)n + \eps n/2$. For any obligate red vertex $v$, by \Cref{lemma:neighborhood_size_guarantees}, part 4, we have $Pr[|N(v,c_v)|\geq Y]\geq 1-\sqrt{p}$, so with probability $1-\sqrt{p} + \sqrt{p'}>0$ there is a color for $v$'s obligate blue partner vertex so that they are connected by an edge and also $|N(v,c_v)|\geq Y$.
    
    We now check that every non-obligate vertex $u$ has a neighborhood which intersects with every neighborhood of size at least $Y$ of the obligate red vertices. Recall that by \Cref{lem:bprimeu}, every $u$ has a neighborhood not robustly intersecting $B$ of size $p^r_u$. Note that we need to show that these neighborhoods intersect outside of $OR\cup OB$ since these vertices may change their colors. To this end, it suffices to show that $(p^r_u-3C) + (Y-3C) > n - |B|$, i.e.,
    $$p^r_u + Y + \beta n - n - 6C \overset{!}{>} 0.$$
    We compute
    $$p^r_u + Y + \beta n - n - 6C \geq \left(\frac{\beta p - \beta + p}{p} + (1-\alpha)(1-p) + \beta - 1 -6\frac{C}{n} + \eps/2\right)n$$
    $$= \left(\beta - \frac{\beta}{p} + 1 - \alpha - p + \alpha p + \beta - 6\frac{C}{n} + \eps/2\right)n$$
    $$ = \left( 1 - p + \beta ( 2 - 1/p) - \alpha (1 - p) - 6\frac{C}{n} + \eps /2\right)n$$
    $$ = \left( 1 - p - (\alpha + \beta)(1-p) - 6\frac{C}{n}+ \eps /2\right)n$$
    $$ \geq \eps n/2 - 6C,$$
    which for large enough $n$ is larger than $0$. \qedhere
\end{proof}

We have covered all possible cases and can thus now prove \Cref{thm:coloring_upper_n_infinity}.

\begin{proof}[Proof of \Cref{thm:coloring_upper_n_infinity}]
    \Cref{lem:noobligates,lem:manyobligates,lem:matching,lem:vertexActingAsObligate,lem:purplevertex,lem:everythingelse} cover all possible cases. Thus, we can always connect the graph, proving the theorem.
\end{proof}

\begin{proof}[Proof of \Cref{thm:main}]
    \Cref{thm:main} follows from \Cref{lem:pairwise_lower,lem:pairwise_upper,lem:edge-subgraph_upper,lem:subgraph_lower,lem:subgraph_upper,lem:coloring_two_colors,lem:coloring_lower,lem:coloring_upper_one_half}, \Cref{cor:matching_bounds,cor:edge-subgraph_lower}, and \Cref{thm:coloring_upper_n_infinity}.
\end{proof}

\clearpage
\bibliographystyle{plainurl}
\bibliography{literature}

\clearpage
\appendix
\section{\texorpdfstring{$G_{n,1/2}$ with Side Constraints}{Gn1/2 with Side Constraints}}\label{app:GnpSideConstraints}

\begin{definition}
    For an integer $k$, let $G^k_{n,1/2}$ be the probability distribution on graphs on $n$ vertices obtained by conditioning $G_{n,1/2}$ on the event that the number of edges in the resulting graph is divisible by $k$.
\end{definition}
Note that, since $G_{n,1/2}$ is the uniform distribution over all graphs on $n$ vertices, $G^k_{n,1/2}$ is the uniform distribution over all graphs on $n$ vertices with number of edges divisible by $k$. We will show that $G^2_{n,1/2}$ is a coloring model, whereas $G^3_{n,1/2}$ is not even pairwise independent.

\begin{lemma}
For any $n \in \mathbb{N}$, the distribution $G^2_{n,1/2}$ is a coloring model.    
\end{lemma}
\begin{proof}
    We prove this statement by providing a coloring model and showing that it is equivalent to $G^2_{n,1/2}$.
    For every vertex $v$, sample a string $s_v \in \{0,1\}^{n-1}$ uniformly at random among all strings with evenly many $1$ bits. The string $s_v$ is the color of $v$. Each bit of $s_v$ corresponds to a distinct other vertex. For vertex $u\neq v$, we will denote by $s_v(u)$ the bit of $s_v$ that corresponds to vertex $u$. The edge $\{u,v\}$ is then present in our graph if and only if $s_v(u) + s_u(v) = 1$. 

    We first show that the graph sampled in this way always has an even number of edges. Note that $\{u,v\}$ is present if and only if $s_v(u) + s_u(v) \equiv 1 \mod 2$. Thus, the parity of the total number of edges is
    $$ \sum_{u \neq v} \big(s_v(u) + s_u(v)\big) \mod 2 = \sum_{v} \sum_{u \neq v} s_v(u) \mod 2 = \sum_u 0 \mod 2 = 0 \mod 2,$$
    as required.

    To see that each graph on evenly many edges is equiprobable, notice that every valid coloring of the vertices is equiprobable, and that each graph on evenly many edges can arise from precisely $ 2^{(n-2)(n-1)/2}$ many different colorings of the vertices. For the latter, consider some graph $G$ on evenly many edges, and let its vertices be $v_1, \dots, v_n$. Recall that $X_{\{u,v\}}$ is the indicator variable that $\{u,v\} \in E(G)$. Notice that in general for any two vertices $u \neq v$, both choices $0$ and $1$ for $s_u(v)$ are possible, but each such choice determines $s_v(u)$ by $s_v(u) = X_{\{u,v\}} - s_u(v) \mod 2$. Thus, there are $2^{n-2}$ valid choices for $s_{v_1}$, each of which is consistent with $G$, because there are $2^{n-2}$ many strings on $n-1$ bits with an even number of $1$ bits. Fixing a color $s_{v_1}$ for $v_1$, this determines uniquely $s_{v_i}(v_1)$ for every $i \in \{2,\dots,n\}$, as noted above. Thus, there are $2^{n-3}$ valid choices for $s_{v_2}$, each of which fixes $s_{v_i}(v_2)$ for each $i \geq 3$. Continuing in the same way, there are $2^{n-i-1}$ valid choices for $s_{v_i}$ for each $i \in [n-1]$ and only one valid choice for $s_{v_n}$. All in all, the number of colorings that produce $G$ is
    $$ \prod_{i=1}^{n-2} 2^i = 2^{\sum_{i=1}^{n-2} i} = 2^{(n-2)(n-1)/2},$$
    which does not depend on $G$, as required.
\end{proof}

\begin{lemma}
    Let $n \geq 4$. Then the distribution $G^3_{n,1/2}$ is not pairwise independent.
\end{lemma}
\begin{proof}
    Consider two vertex-disjoint edges $e$ and $f$. It suffices to show that 
    $$ Pr[X_e|X_f] \neq Pr[X_e|\overline{X_f}],$$
    as this would imply that
    $$ Pr[X_e] = Pr[X_e|X_f] Pr[X_f] + Pr[X_e|\overline{X_f}] Pr[\overline{X_f}] \neq Pr[X_e|X_f].$$
    Let $G \sim G^3_{n,1/2}$ and let $s \in \{0,1\}^{\binom{n}{2}}$ describe the edges in $G$. In particular, choose some order $e_1, \dots, e_m$ on all $m = \binom{n}{2}$ potential edges of $G$ such that $e_1 = f$ and $e_2=e$, and let $s_i = 1$ iff $e_i \in E(G)$. Denote by $w(s)$ the number of $1$ bits in the string $s$, and for any $r \in \{0,1,2\}$ and $k$, let $$S^r_k = \{s \in \{0,1\}^k | w(s) \equiv r \mod 3\}.$$ Then, sampling $G \sim G^3_{n,1/2}$ is equivalent to sampling $s$ uniformly at random from $S^0_m$. We have
    $$ Pr[X_e | X_f] = \frac{Pr[X_e \textrm{ and } X_f]}{ Pr[X_f]} = \frac{\lvert \{s \in S^0_m | s_1 = 1 \textrm{ and } s_2 = 1 \} \rvert }{\lvert \{s \in S^0_m | s_1 = 1 \} \rvert} $$ $$=\frac{\lvert \{s \in S^0_m | s_1 = 1 \textrm{ and } s_2 = 1 \} \rvert }{\lvert \{s \in S^0_m | s_1 = 1 \textrm{ and } s_2 = 1\} \rvert + \lvert \{s \in S^0_m | s_1 = 1 \textrm{ and } s_2 = 0\} \rvert} = \frac{|S^1_{m-2}|}{|S^1_{m-2}| + |S^2_{m-2}|} $$
    and 
    $$ Pr[X_e | \overline{X_f}] = \frac{Pr[X_e \textrm{ and } \overline{X_f}]}{Pr[\overline{X_f}]} = \frac{\lvert \{s \in S^0_m | s_1 = 0 \textrm{ and } s_2 = 1 \} \rvert }{\lvert \{s \in S^0_m | s_1 = 0 \} \rvert} $$
    $$ = \frac{\lvert \{s \in S^0_m | s_1 = 0 \textrm{ and } s_2 = 1 \} \rvert }{\lvert \{s \in S^0_m | s_1 = 0 \textrm{ and } s_2 = 1 \} \rvert + \lvert \{s \in S^0_m | s_1 = 0 \textrm{ and } s_2 = 0 \} \rvert} = \frac{|S^2_{m-2}|}{|S^2_{m-2}| + |S^0_{m-2}|}.$$
    Thus, it suffices to show that 
    $$ \frac{|S^1_{m-2}|}{|S^1_{m-2}| + |S^2_{m-2}|} \overset{!}{\neq} \frac{|S^2_{m-2}|}{|S^2_{m-2}| + |S^0_{m-2}|} $$
    $$ |S^1_{m-2}| |S^2_{m-2}| + |S^1_{m-2}| |S^0_{m-2}| \overset{!}{\neq} |S^2_{m-2}| |S^1_{m-2}| + |S^2_{m-2}| |S^2_{m-2}|  $$
    \begin{align}
    \label{eq:enough_to_show_mod_3}
    |S^1_{m-2}| |S^0_{m-2}| \overset{!}{\neq}|S^2_{m-2}| |S^2_{m-2}|.
    \end{align}
    Note that $|S^0_{m-2}| + |S^1_{m-2}| + |S^2_{m-2}| = 2^{m-2} \neq 0 \mod 3$, implying that it cannot be the case that $|S^0_{m-2}| = |S^1_{m-2}| = |S^2_{m-2}|$.

    If $n$ or $n-1$ is divisible by $3$, then so is $m=\binom{n}{2}$, and so $m-2 \equiv 1 \mod 3$. The remaining case is that $n \equiv 2 \mod 3$, which implies that $m=\binom{n}{2} \equiv 1 \mod 3$, so $m-2 \equiv 2 \mod 3$.

    If $m-2 \equiv 1 \mod 3$, there is a bijection $f:S^0_{m-2} \rightarrow S^1_{m-2}$ obtained by flipping all the bits, that is, for $s \in S^0_{m-2}$, we have $\big(f(s)\big)_i = 1-s_i$ for each $i \in [m-2]$. Thus $|S^0_{m-2}| = |S^1_{m-2}|$ and so \Cref{eq:enough_to_show_mod_3} holds as long as $|S^1_{m-2}| \neq |S^2_{m-2}|$, which we know to be the case since not all $|S^r_{m-2}|$ can be the same.

    Similarly, if $m-2 \equiv 2 \mod 3$, there is a bijection from $S^0_{m-2}$ to $S^2_{m-2}$ obtained by flipping all the bits. So $|S^0_{m-2}| = |S^2_{m-2}|$ and again \Cref{eq:enough_to_show_mod_3} holds as long as $|S^1_{m-2}| \neq |S^2_{m-2}|$, which again has to be the case.
\end{proof}

\section{Proof of Lemma \ref{lem:convexcombinations}}\label{app:convexcombproof}

\begin{proof}[Proof of \Cref{lem:convexcombinations}]
    Let $F:= \alpha \cdot C + (1-\alpha) \cdot D$ be the convex combination of $C$ and $D$ we are considering.
    Note first that for any edge $e$ we have $Pr_F[X_e] = \alpha Pr_C[X_e] + (1-\alpha) Pr_D[X_e] = Pr_C[X_e] = Pr_D[X_e]$.

    Suppose first that $C$ and $D$ are pairwise independent. Consider a pair of non-adjacent edges $e,f$. We have
    $$ Pr_F[X_e \textrm{ and } X_f] = \alpha Pr_C[X_e \textrm{ and }X_f] + (1-\alpha)Pr_D[X_e \textrm{ and } X_f] $$
    $$ = \alpha Pr_C[X_e] Pr_C[X_f] + (1-\alpha) Pr_D[X_e] Pr_D[X_f] = Pr_F[X_e] Pr_F[X_f],$$
    where the second equality follows from the pairwise independence of $C$ and $D$, and the third one follows by the fact that $Pr_F[X_e]=Pr_D[X_e]=Pr_C[X_e]$ for every edge $e$, as noted above. This shows that $F$ is also pairwise independent.

    Now suppose that $C$ and $D$ are both matching independent. To show that $F$ is also matching independent, it suffices to show that for any matching $M$ consisting of vertex-disjoint edges $e_1, \dots, e_k$ and any subset $S \subseteq M$, we have $Pr_F[X_S \textrm{ and } Y_{M \setminus S}] = \prod_{e \in S} Pr_F[X_e] \prod_{e \in M \setminus S} Pr_F[\overline{X_e}]$, where $X_S$ is the event that all the edges in $S$ are present and $Y_{M \setminus S}$ is the event that none of the edges in $M\setminus S$ are present. We have
    $$ Pr_F[X_S \textrm{ and } Y_{M\setminus S}] = \alpha Pr_C[X_S \textrm{ and } Y_{M\setminus S}] + (1-\alpha) Pr_D[X_S \textrm{ and } Y_{M\setminus S}] $$
    $$ = \alpha \prod_{e \in S} Pr_C[X_e] \prod_{e \in M \setminus S} Pr_C[\overline{X_e}] + (1-\alpha) \prod_{e \in S} Pr_D[X_e] \prod_{e \in M \setminus S} Pr_D[\overline{X_e}]$$
    $$ \prod_{e \in S} Pr_F[X_e] \prod_{e \in M \setminus S} Pr_F[\overline{X_e}],$$
    where again we used the matching independence of $C$ and $D$ and the fact that $Pr_F[X_e] = Pr_D[X_e] = Pr_C[X_e]$.

    Finally, suppose $C$ and $D$ are edge-subgraph independent. Let $e=\{u,v\}$ be an edge and $W$ be a vertex set with $u,v \notin W$, and let $H_W$ be a choice for the induced subgraph on $W$. We will show that $F$ is edge-subgraph independent as well, that is,
    $$ Pr_F[G[W] = H_W \textrm{ and } X_e] = Pr_F[G[W] = H_W] Pr_F[X_e].$$
    We have
    $$ Pr_F[G[W] = H_W \textrm{ and } X_e] = \alpha Pr_C[G[W] = H_W \textrm{ and } X_e] + (1-\alpha) Pr_D[G[W] = H_W \textrm{ and } X_e] $$
    $$ = \alpha Pr_C[G[W] = H_W] Pr_C[X_e] + (1-\alpha) Pr_D[G[W] = H_W] Pr_D[X_e] $$
    $$ = \Big(\alpha Pr_C[G[W] = H_W] + (1-\alpha) Pr_D[G[W] = H_W]\Big) Pr_F[X_e] = Pr_F[G[W] = H_W] Pr_F[X_e],$$
    as required, using the edge-subgraph independence of $C$ and $D$.
\end{proof}

\section{\texorpdfstring{Proof of \Cref{claim:CMmodel}}{Proof of Claim \ref{claim:CMmodel}}}\label{app:CModelProof}
\begin{proof}[Proof of \Cref{claim:CMmodel}]
    We fix some potential edge $e=\{u,v\}$. We then have
    $$ p = Pr[X_e] = q \frac{n-2}{n} + (1-q)\frac{1}{n-1},$$
    since in the clique regime, $e$ is not present only if $u$ or $v$ is chosen as the isolated vertex $x$, and in the matching regime, there are $n-1$ choices for a vertex for $u$ to be matched with, each occurring with the same probability. Now consider two potential vertex-disjoint edges $e=\{u,v\}$ and $f = \{w,z\}$. Similarly, we have 
    $$ Pr[X_e \textrm{ and } X_f] = q \frac{n-4}{n} + (1-q) \frac{1}{n-1}\frac{1}{n-3},$$
    where the last term comes from the fact that in the matching regime, conditioned on $u$ and $v$ being matched, the probability of $w$ and $z$ being matched is $\frac{1}{n-3}$.

    In order to have pairwise independence, it suffices to find $q(n)$ such that $Pr[X_e \textrm{ and } X_f] = Pr[X_e] Pr[X_f]$, that is
    \begin{align*}
        \Big( q \frac{n-2}{n} + (1-q)\frac{1}{n-1} \Big)^2 &\overset{!}{=} q \frac{n-4}{n} + (1-q) \frac{1}{n-1}\frac{1}{n-3} \\
        \Bigg( q \Big(\frac{n-2}{n} - \frac{1}{n-1}\Big) + \frac{1}{n-1} \Bigg)^2 &\overset{!}{=} q \Big( \frac{n-4}{n} - \frac{1}{n-1}\frac{1}{n-3} \Big) + \frac{1}{n-1}\frac{1}{n-3}
    \end{align*}
    \begin{multline*}
    q^2 \Big(\frac{n-2}{n} - \frac{1}{n-1}\Big)^2 + 2q\Big( \frac{n-2}{n} - \frac{1}{n-1} \Big)\frac{1}{n-1} + \frac{1}{(n-1)^2} \\
    \overset{!}{=} q \Big( \frac{n-4}{n} - \frac{1}{n-1}\frac{1}{n-3} \Big) + \frac{1}{n-1}\frac{1}{n-3}      
    \end{multline*}

    $$ q^2 \Big(\frac{n^2 - 4n + 2}{n(n-1)}\Big)^2 + q\frac{-n^4 + 11n^3 - 40n^2 + 58n - 24}{n(n-1)^2(n-3)} - \frac{2}{(n-1)^2(n-3)} \overset{!}{=} 0$$
    $$ q^2 (n^5 - 11n^4 +44n^3 - 76n^2 + 52n - 12) + q(-n^5 + 11n^4 - 40n^3 + 58n^2 - 24n) - 2n^2 \overset{!}{=} 0.$$
    Solving the quadratic equation for $q$ with
    \begin{align*}
        a &= n^5 - 11n^4 +44n^3 - 76n^2 + 52n - 12 \\
        b &= -n^5 + 11n^4 - 40n^3 + 58n^2 - 24n \\
        c &= - 2n^2,
    \end{align*}
    we get
    $$ q_{\pm} = \frac{-b \pm \sqrt{b^2 - 4ac}}{2a} $$
    $$ =\frac{n^5 - 11n^4 + 40n^3 - 58n^2 + 24n\pm \sqrt{n^{10} - 22 n^9 + 201 n^8 - 988 n^7 + \Theta(n^6)}}{2n^5 - 22n^4 + 88n^3 - 152n^2 + 104n - 24}. $$
    Consider
    \begin{align*}
        q_+ &=\frac{n^5 - 11n^4 + 40n^3 - 58n^2 + 24n + \sqrt{n^{10} - 22 n^9 + 201 n^8 - 988 n^7 + \Theta(n^6)}}{2n^5 - 22n^4 + 88n^3 - 152n^2 + 104n - 24} \\
        &= \frac{1}{2} - \frac{2}{n^2} + \Theta\Big(\frac{1}{n^3}\Big) + \frac{\sqrt{(n^5 - 11n^4 + 40n^3)^2 -108n^7 + \Theta(n^6)}}{2n^5 - 22n^4 + 88n^3 - 152n^2 + 104n - 24} \\
        &\leq \frac{1}{2} - \frac{2}{n^2} + \Theta\Big(\frac{1}{n^3}\Big) + \frac{1}{2} - \frac{2}{n^2} + \Theta\Big(\frac{1}{n^3}\Big) = 1 - \frac{4}{n^2} + \Theta\Big(\frac{1}{n^3}\Big).
    \end{align*}
    Note also that 
    \begin{align*}
    q_+ &=\frac{n^5 - 11n^4 + 40n^3 - 58n^2 + 24n + \sqrt{n^{10} - 22 n^9 + 201 n^8 - 988 n^7 + \Theta(n^6)}}{2n^5 - 22n^4 + 88n^3 - 152n^2 + 104n - 24}\\
    &= \frac{1}{2} - \frac{2}{n^2} + \Theta\Big(\frac{1}{n^3}\Big) + \frac{\sqrt{(n^5 - 11n^4 + 40n^3-200n^2)^2 + 292n^7 + \Theta(n^6)}}{2n^5 - 22n^4 + 88n^3 - 152n^2 + 104n - 24} \\
    &\geq \frac{1}{2} - \frac{2}{n^2} + \Theta\Big(\frac{1}{n^3}\Big) + \frac{1}{2} - \frac{2}{n^2} + \Theta\Big(\frac{1}{n^3}\Big) = 1 - \frac{4}{n^2} + \Theta\Big(\frac{1}{n^3}\Big).
    \end{align*}
    This shows that for some $q$ with $q = 1 - \frac{4}{n^2} + \Theta\Big(\frac{1}{n^3}\Big)$, the distribution $CM(n,q)$ is well-defined and pairwise independent. Then
    $$ p = Pr[X_e] = q \frac{n-2}{n} + (1-q) \frac{1}{n-1}$$
    $$ = 1 - \Theta \Big( \frac{1}{n^2} \Big) - \frac{2}{n} + \Theta \Big( \frac{1}{n^3} \Big) = 1 - \frac{2}{n} - \Theta \Big( \frac{1}{n^2} \Big),$$
    as required.
    
    Finally, we show that for this $q$, we have that $CM(n,q)$ is not matching independent. For this, consider a matching of three edges $e,f,g$. We have
    $$ Pr[X_e \textrm{ and } X_f \textrm{ and } X_g] = q \frac{n-6}{n} + (1-q)\frac{1}{n-1} \frac{1}{n-3} \frac{1}{n-5} $$
    $$= 1 - \frac{6}{n} - \frac{4}{n^2} + \Theta \Big( \frac{1}{n^3}\Big).$$
    In order to have matching independence, this expression must equal
    $$ Pr[X_e] Pr[X_f] Pr[X_g] = \Bigg( q\frac{n-2}{n} + (1-q)\frac{1}{n-1} \Bigg)^3 $$
    $$ = \Big(1-\frac{2}{n}\Big)^3 \Big(1- \frac{4}{n^2}\Big)^3 + \Theta \Big( \frac{1}{n^3}\Big) = \Big(1-\frac{6}{n} + \frac{12}{n^2}\Big) \Big(1- \frac{12}{n^2}\Big) + \Theta \Big( \frac{1}{n^3}\Big) $$
    $$ = 1 - \frac{6}{n} + \Theta\Big( \frac{1}{n^3} \Big) > 1 - \frac{6}{n} - \frac{4}{n^2} + \Theta \Big( \frac{1}{n^3}\Big) = Pr[X_e \textrm{ and } X_f \textrm{ and } X_g],$$
    thus showing that $CM(n,q)$ for this choice of $q$ is not matching independent. 
\end{proof}

\section{\texorpdfstring{Proof of \Cref{lemma:neighborhood_size_guarantees}}{Proof of Lemma \ref{lemma:neighborhood_size_guarantees}}}\label{app:calculations}
\begin{proof}[Proof of \Cref{lemma:neighborhood_size_guarantees}]
For part 1 of the lemma, we determine the maximum neighborhood size $T$ of each obligate blue vertex $u$ we can guarantee to occur with probability at least $1-p$. The upper bound on the size of the neighborhood is $\beta n$ and the expected size is at least $(p+\eps/2) n$. Thus, by \Cref{lem:markovderivate}, for every $T$ such that
\begin{equation}
\frac{pn + \eps n /2 - T}{\beta n - T}\geq 1-p \label{eqn:Tneeded1}
\end{equation}
we get 
\[Pr[|N(u,c_{u})|>T]\geq 1-p.\]
We solve \Cref{eqn:Tneeded1} for $T$:
   \begin{align*}
        T&\leq \frac{pn + \eps n/2-\beta n(1-p)}{p} \\
        T&\leq \frac{p-\beta + \beta p}{p}n + \frac{\eps n/2 }{p}
    \end{align*}
In particular we can thus guarantee a neighborhood of size $\frac{p-\beta + \beta p}{p}n + \eps n/2 < \frac{p-\beta + \beta p}{p}n + \frac{\eps n}{2p}$ with probability at least $1-p$.
We can show part 2 of the lemma analogously to part 1, replacing $\beta$ with~$\alpha$.

For part 3 of the lemma, we once again use \Cref{lem:markovderivate} to determine the maximum neighborhood size $S$ of $u$ we can guarantee to occur with probability at least $1-\sqrt{p}$. We can use the upper bound of $\beta n$ since $u$ is obligate blue, and the expectation is at least $(p+\eps/2)n$. Thus again, we get
\[Pr[|N(u,c_u)|>S]\geq 1-\sqrt{p} \]
for every $S$ such that
\[\frac{pn + \eps n/2 - S}{\beta n - S}\geq 1-\sqrt{p}, \]
or equivalently
\[S\leq \frac{p n + \eps n /2 -\beta n(1-\sqrt{p})}{\sqrt{p}}\]
\[S\leq \frac{p(1-\beta)n + \eps n / 2}{1-p} = (1-p)(1-\beta)n + \frac{\eps n/2}{1-p},\]
where we used that $\sqrt{p} = 1-p$ by \Cref{lemma:nice-identities}.
In particular we can thus guarantee a neighborhood of size $(1-p)(1-\beta)n+\eps n/2<(1-p)(1-\beta)n+\frac{\eps n/2}{1-p}$ with probability at least $1-\sqrt{p}$. Finally, part 4 is shown analogously to part 3, replacing $\beta$ with $\alpha$.    
\end{proof}

\section{Finitely Many Colors Suffice for the Coloring Model}\label{app:finitely_many_colors}
\begin{proof}[Proof of \Cref{lem:finitely_many_colors}]
Let $\calD$ be a random graph distribution given as a coloring model. We will reduce the number of colors for each vertex one by one, preserving the probability distribution~$\calD$. By a slight abuse of notation, here we refer to any outcome of the experiment at some vertex as its \emph{color}, even if there are infinitely many of them. Suppose we’re considering vertex $v$ and for each graph $H$ on $n$ vertices and color $c$ of $v$, let $p_{H,c}$ be the probability that $H$ is sampled from $\calD$ conditioned on vertex $v$ having color $c$, and let $p_H$ be the unconditional probability of $H$ being sampled from $\calD$. Now for each color $c$ of $v$, consider the vector $\vec{p}_{*,c}$ in $2^{\binom{n}{2}}$-dimensional space that consists of all the $p_{H,c}$ for all possible graphs $H$ on $n$ vertices in some canonical order. Consider also the vector $\vec{p}_*$ of the same dimension with the $p_H$ for all possible graphs $H$ on $n$ vertices as entries in the same canonical order. By the law of total probability, $\vec{p}_*$ is contained in the convex hull of all the $\vec{p}_{*,c}$ vectors (of which there may be infinitely many). It follows by Carathéodory’s theorem~\cite{caratheodory1907variabilitatsbereich} that $\vec{p}_*$ is a convex combination of some $2^{\binom{n}{2}}+1$ many vectors among $\{\vec{p}_{*,c}\;|\;c\text{ is a color of vertex }v\}$. Thus we can pick the corresponding $2^{\binom{n}{2}} +1$ colors for $v$, use the coefficients given by that convex combination as probabilities for these colors in the experiment at vertex $v$, and we end up with the same distribution $\calD$. We repeat this process for each vertex, completing the proof.
\end{proof}

\end{document}

%% file: macros-arxiv.tex
\usepackage[utf8]{inputenc}
\usepackage{xspace}
\usepackage{hyperref}
\usepackage{enumitem}
\usepackage{amsmath,amssymb,amsfonts,amsthm,mathrsfs}
\usepackage{mathtools}
\usepackage{todonotes}
\usepackage{cleveref}
\usepackage{nicefrac}
\usepackage{authblk}
\usepackage{thm-restate}
\usepackage{subcaption}
\usepackage{comment}

\newtheorem{theorem}{Theorem}

\newtheorem{corollary}[theorem]{Corollary}
\newtheorem{lemma}[theorem]{Lemma}
\newtheorem{proposition}[theorem]{Proposition}

\newtheorem{observation}[theorem]{Observation}
\newtheorem{claim}[theorem]{Claim}
\theoremstyle{definition}
\newtheorem{definition}[theorem]{Definition}

\crefname{claim}{claim}{claims}
\crefname{observation}{observation}{observations}

\newcommand{\erdos}{Erd\H{o}s\xspace}
\newcommand{\renyi}{Rényi\xspace}
\newcommand{\lovasz}{Lov\'{a}sz\xspace}
\newcommand{\turan}{Tur\'{a}n\xspace}

\newcommand{\iu}{{i\mkern1mu}}

\newcommand{\E}{\ensuremath{\mathbb{E}}}
\renewcommand{\epsilon}{\varepsilon}
\newcommand{\eps}{\epsilon}
\newcommand{\N}{\ensuremath{\mathbb{N}}}

\newcommand{\calD}{\ensuremath{\mathcal{D}}}
\newcommand{\calG}{\ensuremath{\mathcal{G}}}